\documentclass[onefignum,onetabnum]{siamart190516}

\usepackage{lipsum}
\usepackage{amsfonts}
\usepackage{graphicx}
\usepackage{epstopdf}
\usepackage{algorithmic}
\ifpdf
  \DeclareGraphicsExtensions{.eps,.pdf,.png,.jpg}
\else
  \DeclareGraphicsExtensions{.eps}
\fi

\newsiamremark{remark}{Remark}
\newsiamremark{hypothesis}{Hypothesis}
\crefname{hypothesis}{Hypothesis}{Hypotheses}
\newsiamthm{claim}{Claim}

\headers{Low-order locking-free multiscale method for isotropic elasticity}{A. T. A. Gomes, W. S. Pereira, F. Valentin}

\title{A low-order locking-free multiscale finite element method for isotropic elasticity\thanks{Submitted to the editors DATE.
\funding{This work was partially funded by the MCTI/RNP-Brazil under the HPC4E Project. The second author was partially supported by the CNPq/Brazil No. 140764/2015-1, by the company Bull Ltda and the SCAC from the French Embassy in Brazil under the CIFRE program. The third author was partially supported by CNPq/Brazil No. 301576/2013-0. This work was authored in part by the National Renewable Energy Laboratory, operated by Alliance for Sustainable Energy, LLC, for the U.S. Department of Energy (DOE) under Contract No. DE-AC36-08GO28308. The views expressed in the presentation do not necessarily represent the views of the DOE or the U.S. Government. The U.S. Government retains and the publisher, by accepting the article for publication, acknowledges that the U.S. Government retains a nonexclusive, paid-up, irrevocable, worldwide license to publish or reproduce the published form of this work, or allow others to do so, for U.S. Government purposes.}}}

\author{\
Ant\^onio Tadeu A. Gomes\thanks{Department of Computational and Mathematical Methods, National Laboratory for Scientific Computing  - LNCC, Av. Get\'ulio Vargas, 333, 25651-070 Petr\'opolis - RJ, Brazil (\email{atagomes@lncc.br})}
\and
Weslley S. Pereira\thanks{Computational Science Center, National Renewable Energy Laboratory - NREL, 15257 Denver West Parkway Golden, CO 80401 US
(\email{weslley.dasilvapereira@nrel.gov})}
\and
Fr\'ed\'eric Valentin\thanks{Department of Computational and Mathematical Methods, National Laboratory for Scientific Computing  - LNCC, Av. Get\'ulio Vargas, 333, 25651-070 Petr\'opolis - RJ, Brazil and Nachos Project-Team, Inria Sophia Antipolis - Méditerranée, 2004 Route des Lucioles, 06902 Valbonne, France
(\email{valentin@lncc.br}, \email{frederic.valentin@inria.fr})}}

\usepackage{amsopn}

\usepackage{amssymb,mathrsfs,stmaryrd}
\usepackage[colorinlistoftodos,prependcaption,textsize=tiny
]{todonotes}
\usepackage{xargs} % Use more than one optional parameter in a new commands
\usepackage{graphbox} % Allows for the use of 'align=c' in includegraphics
\usepackage{upgreek} % upepsilon

\usepackage{commath}
\usepackage{siunitx} % \si{} and \num{}

\usepackage{tikz}
\usetikzlibrary{calc}

\usepackage[utf8]{inputenc}
\newcommand\Div{{\nabla\cdot\,}}
\newcommand\bdiv{{\Div}}
\newcommand\ess{{\operatorname{ess}}}

\newcommand\id{I}

\newcommand\dO{{\partial\Omega}}
\newcommand\dK{{\partial K}}

\newcommand\OO{{\Omega}}

\newcommand{\triH}{\calP}

\newcommand{\trih}{{\mathcal{T}_{h}}}

\newcommand{\edgeH}{\mathcal{E}_H}

\newcommand{\sumKP}{\sum_{K\in\calP}}
 
\newcommand\Hbdiv[1]{{H(\boldsymbol{\operatorname{div}}; #1)}}

\newcommand\RR{{\mathbb R}}
\newcommand\NN{{\mathbb N}}

\newcommand\PP{{\mathbb P}}

\newcommand\xx{{\boldsymbol x}}

\newcommand\vv{{\boldsymbol v}}
\newcommand\ww{{\boldsymbol w}}
\newcommand\vvrm{{\boldsymbol v}^{\mathrm{rm}}}
\newcommand\uu{{\boldsymbol u}}
\newcommand\uurm{{\boldsymbol u}^{\mathrm{rm}}}

\newcommand\nn{{\boldsymbol n}}
\newcommand\ff{{\boldsymbol f}}

\newcommand\bq{\boldsymbol q}
\newcommand\bg{{\boldsymbol g}}
\newcommand\bm{\boldsymbol \mu}

\newcommand\bl{\boldsymbol \lambda}

\newcommand\btau{{\boldsymbol \tau}}
\newcommand\HH{{\boldsymbol H}}
\newcommand\LL{{\boldsymbol L}}

\newcommand\bL{{\boldsymbol{\Lambda}}}

\newcommand\EE{{\boldsymbol{\underline\varepsilon}}}

\newcommand\stress{{\boldsymbol \sigma}}
\newcommand\fStress{{\boldsymbol {\underline\sigma}}}

\newcommand\su{{\EE(\uu)}}

\newcommand\suh{{\EE(\uu_h)}}

\newcommand\sv{{\EE(\vv)}}

\newcommand\svh{{\EE(\vv_h)}}

\newcommand\VV{\mathbf V}
\newcommand\VVRM{{\mathbf V}_{\mathrm{rm}}}
\newcommand\tildeVV{\widetilde{\mathbf V}}

\newcommand\calE{\mathcal{E}}
\newcommand\calH{\mathcal{H}}

\newcommand\calP{\mathscr{P}}

\newcommand\calN{\mathscr{N}}
\newcommand\calT{\mathcal{T}}
\newcommand\bK{\partial K}

\newcommandx{\unsure}[2][1=]{\todo[linecolor=red,backgroundcolor=red!25,bordercolor=red,#1]{#2}}
\newcommandx{\change}[2][1=]{\todo[linecolor=blue,backgroundcolor=blue!25,bordercolor=blue,#1]{#2}}
\newcommandx{\info}[2][1=]{\todo[linecolor=green,backgroundcolor=green!25,bordercolor=green,#1]{#2}}
\newcommandx{\improvement}[2][1=]{\todo[linecolor=Plum,backgroundcolor=Plum!25,bordercolor=Plum,#1]{#2}}
\newcommandx{\thiswillnotshow}[2][1=]{\todo[disable,#1]{#2}}
\graphicspath{{img/}}
\newcommand{\mincludegraphics}[1]{\includegraphics[width=.48\linewidth]{#1}}

\newcounter{assumptionCounter}
\theoremstyle{plain}

\crefname{assumption}{Assumption}{Assumptions}
\Crefname{assumption}{Assumption}{Assumptions}
\creflabelformat{assumption}{#2(#1)#3}
\ifpdf
\hypersetup{
  pdftitle={A low-order locking-free multiscale finite element method for isotropic elasticity},
  pdfauthor={A. T. A. Gomes, W. S. Pereira, F. Valentin}
}
\fi

\begin{document}

\maketitle

% REQUIRED
\begin{abstract}
The multiscale hybrid-mixed (MHM) method consists of a multi-level strategy to approximate the solution of boundary value problems with heterogeneous coefficients.
In this context, we propose a family of low-order finite elements for the linear elasticity problem which are free from Poisson locking.
The finite elements rely on face degrees of freedom associated with multiscale bases obtained from local Neumann problems with piecewise polynomial interpolations on faces.
We establish sufficient refinement levels on the fine-scale mesh such that the MHM method is well-posed, optimally convergent under local regularity conditions, and locking-free.
Two-dimensional numerical tests assess theoretical results.
\end{abstract}

% REQUIRED

% REQUIRED

\section{Introduction}
\label{sec:intro}

The Multiscale Hybrid-Mixed (MHM) methods are upscaling numerical strategies to solve boundary value problems on coarse partitions \cite{mhmDarcyAnalisys,HarVal16}.
As in other multiscale finite element methods, the MHM methods use localized multiscale basis functions to recover structures of the solution lost by the unresolved fine scales.
From a computational viewpoint, this class of methods fit well in massively parallel computer systems because they allow for the multiscale basis functions to be computed in a decoupled fashion.
The global degrees of freedom in the MHM methods localize on the faces of the mesh skeleton, and convergence can be achieved without refining the global partition \cite{BarJaiParVal20,10.1093/imanum/drac041}.
%The discontinuous interpolations on faces play an essential role in accuracy in such a crossing-face interface scenario.

The MHM method was introduced to the two- and three-dimensional linear elasticity model in \cite{Harder2016} using polynomial interpolations on faces and in the second-level discretization.
This work was recently extended to global polytopal partitions and new finite elements \cite{10.1093/imanum/drac041}, encompassing discontinuous interpolations on faces and refined local meshes.
Both works use the countinuous Galerkin method on top of the displacement elasticity formulation to build their multiscale basis functions.
The MHM-Hdiv method \cite{DevFarGomPerSanVal19} uses the same global problem from \cite{Harder2016} but, in contrast, its local problems use mixed finite elements to recover a global Hdiv-conforming numerical stress tensor.
%Although both methodologies work on general polygonal meshes and can achieve convergence without refining the global partition, the construction and analysis of the former MHM method and the MHM-Hdiv are fundamentally different.
The MHM method proposed in \cite{Pereira2016} uses the same interpolation spaces from \cite{10.1093/imanum/drac041} but a different local-level solver that deals with quasi-incompressible isotropic materials.
This last work lacks proper stability and convergence analyses, and the locking-free property validation.
We shall highlight that all four examples use the same global level problem and the choice of local level solver depend on the problem one wants to solve.

Inf-sup stability of most of the MHM methods presented in the literature relies on high-order polynomial spaces in the local level, e.g., \cite{mhmDarcyAnalisys,Araya2017,Duran2019,ValEtAl17mhm,Harder2016,mhmAdveccao,LatParSchVal18}.
A new way to prove stability was introduced in \cite{BarJaiParVal20} for the two-dimensional Poisson equation.
It encompasses the low-order global-local pairs $(\ell,\ell)$ and $(\ell,\ell+1)$, where $\ell$ is the polynomial degree used in the first-level solver.
Low order finite elements are appealing for multiscale problems since they are computationally cheaper option for low-regularity problems.

It is well-known that standard low-order finite element methods applied to the displacement formulation of nearly incompressible elasticity problems result in poor observed convergence rates \cite{lockingBabuska}.
In the $h$-version of the continuous Galerkin method using piecewise linear polynomials on triangular meshes, the theoretical convergence rate only hold for $h$ sufficiently small.
This phenomenon is commonly known as locking (or Poisson locking).
One way to circumvent this issue is to rewrite the elasticity problem in its mixed version, and approximate stress and displacement fields separately.
Another possibility is to introduce an extra pressure variable, mimicking what is done in the Stokes problem.
Since the choice of pairs of inf-sup stable approximation spaces is non-trivial in both cases, a classical approach to overcome this limitation employs stabilized schemes \cite{stabilizedElasticity}.
% Other methods to deal with nearly incompressible materials are \cite{??}.

In this paper, we provide the stability and a priori convergence analysis for the family of finite elements proposed in \cite{Pereira2016}.
The stability of the MHM method is based on the low-order global-local compromises $(\ell,\ell)$ and $(\ell,\ell+1)$ following closely \cite{BarJaiParVal20}.
We prove optimal mesh-based convergence for displacement, pressure and traction approximations under local regularity assumptions.
We show the resulting MHM method is locking-free in the sense that the stability and a convergence constants do not degenerate when the Poisson ratio approaches $1/2$.
The second-level solver uses a Least Squares stabilization for the Galerkin method \cite{stabilizedElasticity}, which possibly adopting the (appealing) equal-order polynomial spaces for displacement and pressure.
We present some analytical numerical tests to verify the theoretical properties.
We find an additional $O(+H/2)$ convergence in the skeleton-based refinement strategy, that was also observed in other families of MHM methods, e.g., \cite{10.1093/imanum/drac041,BarJaiParVal20}.

Using a problem with nearly incompressible materials, we verify the MHM method from \cite{10.1093/imanum/drac041,Pereira2016} improves the robustness of the continuous Galerkin method using compatible configurations between them.
The same example shows how the locking-free finite elements present here and the Galerkin Least Squares method from \cite{stabilizedElasticity} solves the Poisson locking issue.
Finally, we test the versatility of the MHM method to solve a heterogeneous elasticity problem in a composite domain with nearly incompressible materials.
We show it is sufficient to use the locking-free finite elements only in the nearly incompressible region to improve overall accuracy of the MHM solver.

The paper's outline is as follows.
In \cref{sec:model}, we present the isotropic elasticity problem in its classical and displacement-pressure forms.
In \cref{sec:mhm}, we revisit the MHM method from \cite{10.1093/imanum/drac041} that handles high-contrast heterogeneous coefficients using a multi-level approach.
In \cref{sec:lockingFree}, we present a family of stabilized MHM methods based on the Least-Squares stabilization of the Galerkin method.
We prove the those methods are well-posed and locking-free. As a subproduct, we show how to adapt the original Galerkin Least Squares method to pure Neumann problems.
We prove the family of methods is optimally convergent under local regularity conditions in \cref{sec:convergenceMHM}.
We dedicate \cref{sec:numResultsElasticity} to the numerical validation of the theory, and to show some numerical estimates of the MHM method.
We present some concluding remarks in \cref{sec:concl}.

\begin{remark}
Above, and hereafter, we adopt the typical function spaces and differential operators \cite{ernGuermondFEM}.
The \textbf{bold} style indicates $d$-dimensional vector spaces, e.g., $\LL^2(\OO) := L^2(\OO)^d$.
\end{remark}

\section{The elasticity problem}
\label{sec:model}

% Elasticity problem with heterogeneities
Let $\OO \subset\RR^d$, $d\in\{2,\,3\}$, be an open and bounded domain with polygonal Lipschitz boundary $\dO = \overline{\Gamma_D \cup \Gamma_N}$, $\Gamma_D \cap \Gamma_N = \emptyset$, and $\Gamma_D \neq \emptyset$.
Consider the elasticity problem of finding a displacement $\uu:\OO\to\RR^d$ that satisfies
%\begin{equation}\label{elliptic}
%\int_\Omega 2\,G\, \left( \su : \sv + \frac{\nu}{1-2\,\nu}\, (\Div\uu)\, (\Div\vv) \right) \dif x
%= \int_\Omega \ff\cdot \vv \dif x\,, \quad\forall \vv\in \HH^1_0(\Omega)\,,
%\end{equation}
\begin{align}
	\label{elliptic}
	- \Div\left(\fStress(\uu)\right) &= \ff \;\text{ in }\; \OO\,,
	\qquad \uu = \mathbf{0} \;\text{ on }\; \Gamma_D
	\quad\text{and}
	\quad \fStress(\uu)\,\nn = \bg \;\text{ on }\; \Gamma_N\,,
\end{align}
where $\ff \in \LL^2(\OO)$ is the distributed load,
$\bg \in \LL^2(\Gamma_N)$ is the traction,
$\nn$ is the outward unit normal vector field defined a.e. on $\dO$,
$\stress := \fStress(\uu)$ is the isotropic stress tensor
\begin{align}
	\label{stressFunc}
	\fStress(\uu) &:= 2\,G\,\left(\su + \frac{\nu}{1-2\,\nu}\, (\Div\uu)\,\id_{d}\right)\,,
\end{align}
$G,\nu\in L^\infty(\OO)$ are the shear modulus and the Poisson’s ratio, that possibly depend on $\OO$, $\id_{d} \in \RR^{d\times d}$ is the identity matrix, and $\su := (\nabla\uu+(\nabla\uu)^t) / 2$ is the infinitesimal strain tensor.
We assume there exist $G_0,\nu_0 \in \RR$ such that $0 < G_0 \leqslant G$ and $0 < \nu_0 \leqslant \nu < 1/2$ a.e. in $\OO$.
Under these assumptions, there exists a unique solution $\uu \in \HH^1(\OO)$ for \cref{elliptic} in a distributional sense via the BNB Theorem (c.f. \cite[Theorem 2.6]{ernGuermondFEM}).

% Poisson locking
Materials with high Poisson's ratio ($\nu \approx 1/2$) are usually referred as nearly incompressible, or quasi-incompressible.
We call \emph{Poisson locking phenomenon} the poor convergence order appearing in a numerical method when used to obtain approximate solutions to linear elasticity problems with nearly incompressible materials.
This phenomenon occurs, for instance, in low-order continuous Galerkin formulations for \cref{elliptic} using piecewise polynomials on triangular and quadrilateral meshes \cite{lockingBabuska}.

% Mixed formulations
A well known technique to avoid the Poisson locking starts by rewriting \cref{elliptic} in an equivalent mixed problem that approximates stress and displacement fields separately.
Alternatively, consider the following displacement-pressure mixed form proposed in \cite{Herrmann1965}
\begin{align}
	\label{mixed}
	\begin{aligned}
	- \Div\left(2\,G\,\su\right) + \Div\left(p\,\id_{d}\right) &= \ff &&\text{ in }\; \OO\,,\\
	\Div\uu + \upepsilon\,p &= 0 && \text{ in }\; \OO\,,\\
	\uu &= \mathbf{0} && \text{ on }\; \Gamma_D\,,\\
	\left(2\,G\,\su - p\,\id_{d}\right)\,\nn &= \bg && \text{ on }\; \Gamma_N\,,	
	\end{aligned}
\end{align}
where $p$ is the (Herrmann) pressure
\begin{align}\label{eq:pAndEpsilon}
	p := -\frac{1}{\upepsilon}\,\Div\uu\,,\qquad\text{and}\quad
	\upepsilon := \frac{1-2\nu}{2\,G\,\nu}\,.
\end{align}
% \wes{Traditional numerical methods applied to \cref{mixed} are (Poisson) locking-free, which is the case of \wes{... [cite]}.
%
From a practical viewpoint, however, obtaining inf-sup stable numerical methods for usual discrete mixed formulations is non-trivial \cite[§8.12.1]{boffi2013mixed}.
A classical approach to overcome the inf-sup limitation of \cref{mixed} that keeps the desirable locking-free property is to employ stabilized schemes \cite{stabilizedElasticity}.
\Cref{sec:GaLS} presents the Galerkin Least Squares (GaLS) method, a stabilized finite element scheme to solve displacement-pressure mixed formulation.

\section{The MHM method}
\label{sec:mhm}

Let $\calP$ be a collection of open and bounded $d$-polytopes $K$, such that $\overline{\OO}=\cup_{K\in\calP}^{} \overline{K}$.
Associated to $K \in \calP$, we define the spaces
\begin{align*}
\VVRM(K) &:= \left\{\vvrm\in \HH^1(K)\,:\, \EE(\vvrm) = 0\right\}\,,\\
\tildeVV(K) &:= \left\{\tilde\vv\in \HH^1(K)\,:\, \int_K \tilde\vv\cdot\vvrm\dif x = 0\,,\quad\forall\, \vvrm\in \VVRM(K)\right\}\,,%\label{eq:tildeVK}
\end{align*}
The functions $\vvrm \in \VVRM(K)$, known as rigid body motions, can be written as $\vvrm(\xx) = \mathbf a + \beta\, (-x_2, x_1)$, if $d=2$, and $\vvrm(\xx) = \mathbf a + \mathbf b \times \xx$, if $d=3$, where $\mathbf a, \mathbf b \in \RR^d$ and $\beta \in \RR$.
Therefore, $\VVRM(K)$ is a finite dimensional space of dimension $d\,(d+1)/2$.
We also define the global spaces associated to $\calP$
\begin{align*}
%\label{spaceBL}
\bL &:= \left\{ \btau\,\nn^K\sVert[0]_{\bK}\,,\;\forall\, K\in\calP \, : \, \btau\in\Hbdiv{\OO}\,,\; \btau\,\nn\sVert[0]_{\Gamma_N} = \mathbf 0\right\}\,,\\
%\label{eq:spaceOfRigidBodyModes}
\VVRM &:= \left\{\vvrm\in \LL^2(\OO)\,:\, \vvrm\sVert[0]_K\in\VVRM(K)\,,\quad\forall\, K\in \calP\right\}\,,\\
%\label{eq:spaceVtilde}
\tildeVV &:= \left\{\tilde\vv\in \LL^2(\OO)\,:\, \vv\sVert[0]_K\in\tildeVV(K)\,,\quad\forall\, K\in \calP\right\}\,,
\end{align*}
where the symbol $\nn^K$ denotes the outward unit normal vector field on the boundary $\dK$. Note the functions $\bm \in \bL$ belong to $\HH^{-\frac{1}{2}}(\partial K)$ for each $K \in \calP$.

We denote by $\Pi_{RM}^{}$ the $\LL^2(\OO)$ projection onto $\VVRM$, i.e.,  given  $\vv \in \HH^1(\calP) := \tildeVV \oplus \VVRM$, the function $\Pi_{RM}^{}\vv\,|_K^{}$ satisfies
\begin{align}\label{eq:projectionVRM}
\int_K \Pi_{RM}\vv \cdot \vvrm\,d\xx = \int_K \vv\cdot\vvrm\,d\xx \quad \text{for all }\vvrm \in \VVRM(K)\,,
\end{align}
%~\\
which immediately leads to the following estimates for all $\vv \in \HH^1(\calP)$
\begin{align}\label{eq:stabilityVsum}
%\begin{aligned}
%\|\Pi_{RM}\vv\|_{0,\OO} &\le \|\vv\|_{0,\OO}\,,\\
\|\vv-\Pi_{RM}\vv\|_{0,\OO} &\le \|\vv\|_{0,\OO}
\quad\text{and}\quad 
\|\EE(\vv-\Pi_{RM}\vv)\|_{0,\calP} \le \|\EE(\vv)\|_{0,\calP}\,.
%\end{aligned}
\end{align}
Hereafter, $\|\cdot\|_{m,D}$ and $|\cdot|_{m,K}$ denote the usual norm and semi-norm, respectively, in the spaces $H^m(D)$, $\HH^m(D)$ and $H^m(D)^{d\times d}$, for $m \in \NN \cap \{0\}$ and $D \subset \RR^d$,
%, i.e.,
%\begin{align*}
%	\|\cdot\|_{0,D} := \left(\int_D |\cdot|^2 \dif x\right)^{\frac12}\,,\qquad
%	\|\cdot\|_{m,D} := \left(\sum_{|\alpha| < m} \|\partial^\alpha \cdot\|_{0,D}^2\right)^{\frac12}\,.
%\end{align*}	
%where $\alpha$ represents a multi-index and $\partial^\alpha$ represents the partial derivative directed by $\alpha$.
%%
and
\begin{align*}
	\|\cdot\|_{m,\calP} := \left(\sumKP \|\cdot\|_{m,K}^{2}\right)^{\frac12}\,,\qquad
	|\cdot|_{m,\calP} := \left(\sumKP |\cdot|_{m,K}^{2}\right)^{\frac12}\,.
\end{align*}	
% the square root of the sum of the squares of the $L^2$-norm of each derivative.
%We use the same notation for the norms in $\HH^m(D)$ and $H^m(D)^{d\times d}$ by taking the square root of the sum of the squares of the $H^m(D)$-norm of each component.
% summing the derivatives each component.
are a norm and a semi-norm in $\HH^1(\calP)$.
Also, we equip $\bL$ with the norm
%$\HH^m(\calP)$ and
%$\bL$ as follows
\begin{equation}\label{eq:bnd_b}
%\|\cdot\|_{m,\calP} := \left(\sumKP \|\cdot\|_{m,K}^{2}\right)^{\frac12}\,,\qquad
\|\bm\|_\bL^{} := \sup_{\vv \,\in\, \HH^1(\calP)\backslash\{\mathbf{0}\} }^{} \frac{\sum_{K\in\calP} \langle \bm,\vv \rangle_{\dK}^{}}{\|\vv\|_{1,\calP}^{}}  \qquad \text{for every } \,\bm\in \bL\,.
\end{equation}
%Notice that $\|\cdot\|_{-\frac12,\calbP}$ is a norm for $\bL$.
%
The notation $\langle\cdot,\cdot\rangle_{\partial K}$ represents the duality pairing between $\HH^{-\frac{1}{2}}(\partial K)$ and 
$\HH^{\frac{1}{2}}(\partial K)$
%, i.e., 
%\begin{align}\label{eq:defDualProdK}
%\langle\bm,\vv\rangle_{\partial K} := \int_K (\nabla\cdot \btau) \cdot \vv \,d\xx +\int_K \btau  : \nabla \vv \,d\xx \quad \text{for all } \vv \in \HH^1(K),
%\end{align}
%where $\btau \in \Hbdiv{\OO}$ with $\btau\nn^K=\bm$ on $\dK$.
in a way that if $\vv \in \HH^1(K)$ and $\bm \in \LL^2(\dK)$ then $\langle\bm,\vv\rangle_{\partial K} = \int_{\dK} \bm\cdot\vv\dif s$.
%In the sequel, we also use \cref{eq:defDualProdK} with $K = \OO$.

We now present the continuous formulation that is basis for the MHM methods.

~
\paragraph{Global-local formulation}

Let $T:\bL\rightarrow \tildeVV$ and $\hat T:\LL^2(\OO)\rightarrow \tildeVV$ be the linear operators such that, on each $K \in \calP$, $T(\bm)\sVert[0]_K$ and $\hat{T}(\bq)\sVert[0]_K$
%are bounded linear operators defined
are the unique solutions in $\tildeVV(K)$ of
\begin{align}
a_K( T(\bm), \tilde\vv )
&= \langle \bm, \tilde\vv\rangle_{\dK} && \text{for all } \bm\in \bL,\; \tilde \vv\in \tildeVV(K),\label{e:pel} \\
a_K( \hat{T}(\bq), \tilde\vv )
&= \int_K \bq\cdot\tilde\vv \dif x + \int_{\dK \cap \Gamma_N} \bg\cdot\tilde\vv \dif s
 && \text{for all } \bq\in \LL^2(\OO),\; \tilde \vv\in \tildeVV(K)\,,\label{e:pef}
\end{align}
respectively, where $\tilde \vv\in \tildeVV(K)$ and
\begin{align*}
	a_K\left( \uu, \vv \right) := \int_K \left( 2\,G\,\su : \sv + \frac1{\upepsilon}\, (\Div\uu)\, (\Div\vv) \right) \dif x
	 && \text{for all } \uu, \vv\in \HH^1(\calP).
\end{align*}
Owing to these definitions, we can rewrite the solution of \cref{elliptic} equivalently as
\begin{equation}\label{decomposition}
\uu = \uurm+T(\bl)+\hat{T}(\ff)\,,
\end{equation}
where $(\bl, \uurm)\in\bL\times \VVRM^{}$ solves the following mixed problem
%\emph{Find $(\bl, \uurm)\in\bL\times \VVRM^{}$  such that }
\begin{align}
\label{hybrid}
 \sum_{K\in\calP} \left[ \langle \bm, T(\bl)\rangle_{\dK} + \langle \bm, \uurm\rangle_{\dK} \right] &= - \sum_{K\in\calP} \langle \bm, \hat T(\ff)\rangle_{\dK} &&\text{for all }\bm\in \bL\,,\\
 \label{hybrid-VVRM}
\sum_{K\in\calP} \langle \bl, \vvrm\rangle_{\dK} &= -\int_{\OO} \ff\cdot\vvrm\dif x &&\text{for all } \vvrm\in \VVRM^{}\,.
\end{align}
%The MHM method is a discretization of the global-local formulation \cref{e:pel}-\cref{hybrid}.
%
The formulation \crefrange{e:pel}{hybrid-VVRM} is known as global-local formulation and is equivalent to \cref{elliptic} in a distributional sense as proved in \cite{10.1093/imanum/drac041}.
The hybridization variable $\bl$ represents the traction vector field along the skeleton of $\calP$, i.e.,
\begin{align}\label{lambda}
	\bl = \stress\nn^K\qquad \text{on}\quad \dK\backslash\Gamma_N\quad \text{for all}\quad K\in\calP\,.
\end{align}
%The global problem \cref{hybrid,hybrid-VVRM} has a unique solution $(\bl, \uurm)\in\bL\times \VVRM^{}$
%
The MHM method for linear elasticity relies on the discretization of \crefrange{e:pel}{hybrid-VVRM}.
%We now briefly present the discrete version of \cref{hybrid}-\cref{hybrid-VVRM} used to build the MHM method.

~
\paragraph{MHM method's discrete formulation}

Let $\bL_H$ be a discrete subspace of $\bL$, and $T_h$ and $\hat T_h$ be linear operators that approximate $T$ and $\hat T$.
The discrete version of \cref{hybrid}-\cref{hybrid-VVRM} is to search for $(\bl_H,\uurm_\calH) \in \bL_H\times\VVRM$ such that
%We are interested in the solution discrete formulation
\begin{align}
 \sum_{K\in\calP} \left[ \langle \bm_H, T_h(\bl_H)\rangle_{\dK} + \langle \bm_H, \uurm_\calH\rangle_{\dK} \right]
%\nonumber\\
 &= - \sum_{K\in\calP} \langle \bm_H, \hat T_h(\ff)\rangle_{\dK} &\text{for all }\bm_H\in \bL_H\,,
\label{hybrid-h}\\
 \label{hybrid-VVRM-h}
\sum_{K\in\calP} \langle \bl_H, \vvrm\rangle_{\dK} &= -\int_{\OO} \ff\cdot\vvrm\dif x &\text{for all } \vvrm\in \VVRM^{}.
\end{align}
The post-processed discrete displacement solution is
\begin{align}
\label{decompositionUh}
\uu_{Hh} := \uurm_\calH + T_h(\bl_H) + \hat T_h(\ff)\,.
\end{align}
and $\bl_H$ is the discrete traction that approximates $\stress\nn^K$ on $\dK\backslash\Gamma_N$ for all $K\in\calP$.
In the sequel, we define $\bL_H$, $T_h$ and $\hat T_h$ so that \cref{hybrid-h}-\cref{hybrid-VVRM-h} is well-posed and $(\bl_H,\uurm_\calH)$ approximates $(\bl,\uurm)$, the solution of \cref{hybrid}-\cref{hybrid-VVRM}.

\begin{remark}
\Cref{e:pel,e:pef} are classical weak formulations for
\begin{align*}
	%\label{elliptic_Tmu}
	- \Div\left(\fStress(\,T(\bm)\,)\right) &= R_{\bm} \;\text{ in }\; K
	&\text{and}\quad
	&\fStress(\,T(\bm)\,)\,\nn^K = \bm \;\text{ on }\; \dK\,,
\end{align*}
and
\begin{align*}
	%\label{elliptic_Tf}
	- \Div\left(\fStress(\,\hat T(\ff)\,)\right) &= \Pi_{RM}(\ff) \;\text{ in }\; K
	&\text{and}\quad
	&\fStress(\,\hat T(\ff)\,)\,\nn^K = \bg \;\text{ on }\; \dK \cap \Gamma_N\,,\\
	&&&\fStress(\,\hat T(\ff)\,)\,\nn^K = \mathbf 0 \;\text{ on }\; \dK\backslash\Gamma_N\,,
\end{align*}
respectively, where $R_{\bm}$ is the unique function in $\VVRM(K)$ satisfying $\int_{K} R_{\bm}\cdot\vvrm\dif x = - \langle \bm, \vvrm\rangle_{\dK}$ for all $\vvrm \in \VVRM(K)$.
These are pure traction elasticity problems similar to \cref{elliptic} with the additional condition $T(\bm)\sVert[0]_K, \hat{T}(\bq)\sVert[0]_K \in \tildeVV(K)$; therefore, they admit a global-local formulation if we repeat the procedure on a partition $\calP_K$ of $K$.
Thus, one may use the MHM method to build multi-level algorithms.
\end{remark}

\begin{remark}
\Cref{hybrid-VVRM-h} states the local equilibrium of the elastic body since $\bl_H$ is the discrete traction field on $\dK$.
\end{remark}

\begin{remark}
The mappings $T_h$ and $\hat T_h$ can be defined quite generally. The choice depends on which unknown one wants to approximate accurately and impacts the robustness of the method.
For instance, to ensure that the method yields a numerical stress tensor $\fStress(\uu_{Hh}) \in H(\boldsymbol{\operatorname{div}}; \OO)$, one may use a stress mixed finite element method to approximate \cref{e:pel,e:pef} \cite{DevFarGomPerSanVal19}.
The continuous Galerkin method with piecewise polynomials on triangular meshes was used in \cite{10.1093/imanum/drac041} to discretize \cref{e:pel,e:pef} and obtain $T_h$ and $\hat T_h$.
The operators $T_h$ and $\hat T_h$ for this work are defined in \cref{sec:GaLS}.
\end{remark}

\section{Locking-free finite elements for the MHM method}
\label{sec:lockingFree}

The MHM method uses a multi-level discretization starting from the first-level partition $\calP$.
In this work, we specify the partitions and spaces used in a two-level version of the method.

\subsection{Preliminaries}

Without loss of generality, we shall use hereafter the terminology employed for three-dimensional domains.
Hereafter, $h_D^{} := \sup_{x,y\in D}^{}|x-y|$ is the diameter of an arbitrary bounded set $D \subset \RR^n$, $n\in\NN$.
The  radius of the largest inscribed ball in $D$ reads $\rho_D$, and the shape regularity of $D$ is denoted by $\sigma_D^{} := h_D^{} / \rho_D^{}$.

Let $\calE$ be the set of the faces in $\calP$, and $\left\{\calE_H^{}\right\}_{H>0}^{}$ be a family of simplicial conformal partitions of $\calE$.
For each $K \in \calP$, let $\left\{\calT_h^{K}\right\}_{h>0}^{}$ be a shape-regular family of local simplicial conformal partitions of $K$, and define $\calT_h^{} := \cup_{K\in\calP}^{} \calT_h^{K}$ for each $h > 0$.
Finally, we state the three characteristic sizes used recurrently hereafter:
\begin{align}
\calH := \max_{K\in\calP}^{} h_K^{}\,,\qquad
H := \max_{F \in \calE_H^{}} h_{F}\,,\qquad
h := \max_{\tau \in \calT_h^{}} h_{\tau}\,.
\end{align}
See \cref{dom2d} for an illustration.
Note that $\calP$ and $\calT_h^{}$ can be nonconformal, which brings versatility in the choice of both global and local partitions.
 
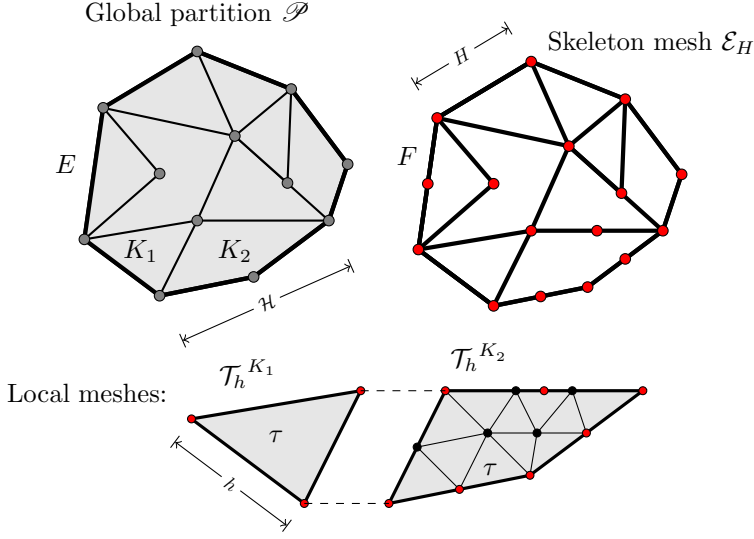
\begin{figure}[!htb]
\centering
\begin{tikzpicture}
	\node (A) at (0,0) {
%
% A global mesh.
%
\begin{tikzpicture}[scale=0.5]
%\draw[->] (-2.0,-2.0) -- (6.0,-2.0) node[right] {$x_1$};
%\draw[->] (-2.0,-2.0) -- (-2.0,6.0) node[above] {$x_2$};
%\draw[orange, ultra thick ] (-1,0) -- (1,0) ;
%===================
\coordinate (A1) at (0.50,-1.50);
\coordinate (A2) at (-1.50,0);
\coordinate (A3) at (-1.0,3.50);
\coordinate (A4) at (1.50,5);
\coordinate (A5) at (4,4);
\coordinate (A6) at (5.50,2);
\coordinate (A7) at (5.0,0.50);
\coordinate (A8) at (3,-1.0);
\coordinate (A9) at (-2,2);
\coordinate (A10) at (2.5,-0.3);
\coordinate (A11) at (1.5,6);
\coordinate (A12) at (0,-0.3);
%==========================
\coordinate (B1) at (0.50,1.75);
\coordinate (B2) at (2.50,2.75);
\coordinate (B3) at (1.50,0.50);
\coordinate (B4) at (3.9,1.50);
%==========================
\draw[fill=gray!20,  ultra thick ] (A1) -- (A2) -- (A3) -- (A4) --
(A5) -- (A6) -- (A7) -- (A8) -- cycle ;
%
%\draw[fill=gray!20,  thick ] (B1) -- (B2) -- (B3) -- cycle ;
%
\draw[fill=gray!20,  thick ] (B1) -- (A3) -- (A2) -- cycle ;
\draw[fill=gray!20,  thick ] (A3) -- (B2) -- (A4) -- cycle ;
%
%\draw[fill=gray!20,  thick ] (A5) -- (B2) -- (A6) -- cycle ;
%
\draw[fill=gray!20,  thick ] (B3) -- (A7) -- (B2) -- cycle ;
%
%\draw[fill=gray!20,  thick ] (B3) -- (A1) -- (A8) -- cycle ;
%
\draw[fill=gray!20,  thick ] (A2) -- (A1) -- (B3) -- cycle ;
\draw[fill=gray!20,  thick ] (B4) -- (A5) -- (B2) -- cycle ;
\draw[black,  ultra thick ] (A1) -- (A2) -- (A3) -- (A4) -- (A5) --
(A6) -- (A7) -- (A8) -- cycle ;
\draw[fill=gray] (A1) circle (4.0 pt);
\draw[fill=gray] (A2) circle (4.0 pt);
\draw[fill=gray] (A3) circle (4.0 pt);
\draw[fill=gray] (A4) circle (4.0 pt);
\draw[fill=gray] (A5) circle (4.0 pt);
\draw[fill=gray] (A6) circle (4.0 pt);
\draw[fill=gray] (A7) circle (4.0 pt);
\draw[fill=gray] (A8) circle (4.0 pt);
%
%========================================
\draw[fill=gray] (B1) circle (4.0 pt);
\draw[fill=gray] (B2) circle (4.0 pt);
\draw[fill=gray] (B3) circle (4.0 pt);
\draw[fill=gray] (B4) circle (4.0 pt);
%==============================
%
\node[rotate=0 ] at (A9) {$E$};
\node[rotate=0 ] at (A12) {$K_1$};
\node[rotate=0 ] at (A10) {$K_2$};
\node[rotate=0 ] at (A11) {Global partition $\calP$};
\draw[|<->|, xshift=0.1cm,yshift=-0.4cm] 
    ($ (A1) + (.6,-1.2) $) -- ($ (A7) + (.6,-1.2) $) node[fill=white, midway,sloped,scale=0.025cm] {$\mathcal{H}$} ;
% \draw[rotate=-45,>=stealth, color=black, very thin] (4.35,4.25) -- (4.35,2.25) ;
%\node[rotate=70] at (B5) {$d(x,z)$};
%
\end{tikzpicture}	
	};
	\node (B) at (5,0) {
%
% A global mesh.
%
\begin{tikzpicture}[scale=0.5]
%\draw[->] (-2.0,-2.0) -- (6.0,-2.0) node[right] {$x_1$};
%\draw[->] (-2.0,-2.0) -- (-2.0,6.0) node[above] {$x_2$};
%\draw[orange, ultra thick ] (-1,0) -- (1,0) ;
%===================
\coordinate (A1) at (0.50,-1.50);
\coordinate (A2) at (-1.50,0);
\coordinate (A3) at (-1.0,3.50);
\coordinate (A4) at (1.50,5);
\coordinate (A5) at (4,4);
\coordinate (A6) at (5.50,2);
\coordinate (A7) at (5.0,0.50);
\coordinate (A8) at (3,-1.0);
\coordinate (A9) at (-1.8,2.5);
\coordinate (A10) at (2.5,-0.3);
\coordinate (A11) at (4.7,5.5);
%==========================
\coordinate (B1) at (0.50,1.75);
\coordinate (B2) at (2.50,2.75);
\coordinate (B3) at (1.50,0.50);
\coordinate (B4) at (3.9,1.50);
%==========================
\coordinate (D1) at (4,-0.25);
\coordinate (D2) at (3.25,0.50);
%\coordinate (D3) at (1,-0.50);
\coordinate (D4) at (1.75,-1.25);
\coordinate (D5) at (-1.25,1.75);
%==========================
\draw[  ultra thick ] (A1) -- (A2) -- (A3) -- (A4) --
(A5) -- (A6) -- (A7) -- (A8) -- cycle ;
%
%\draw[  ultra thick ] (B1) -- (B2) -- (B3) -- cycle ;
%
\draw[  ultra thick ] (B1) -- (A3) -- (A2) -- cycle ;
\draw[  ultra thick ] (A3) -- (B2) -- (A4) -- cycle ;
%
%\draw[  ultra thick ] (A5) -- (B2) -- (A6) -- cycle ;
%
\draw[  ultra thick ] (B3) -- (A7) -- (B2) -- cycle ;
%
%\draw[  ultra thick ] (B3) -- (A1) -- (A8) -- cycle ;
%
\draw[  ultra thick ] (A2) -- (A1) -- (B3) -- cycle ;
\draw[  ultra thick ] (B4) -- (A5) -- (B2) -- cycle ;
\draw[black,  ultra thick ] (A1) -- (A2) -- (A3) -- (A4) -- (A5) --
(A6) -- (A7) -- (A8) -- cycle ;
\draw[fill=red] (A1) circle (4.0 pt);
\draw[fill=red] (A2) circle (4.0 pt);
\draw[fill=red] (A3) circle (4.0 pt);
\draw[fill=red] (A4) circle (4.0 pt);
\draw[fill=red] (A5) circle (4.0 pt);
\draw[fill=red] (A6) circle (4.0 pt);
\draw[fill=red] (A7) circle (4.0 pt);
\draw[fill=red] (A8) circle (4.0 pt);
%
%========================================
\draw[fill=red] (B1) circle (4.0 pt);
\draw[fill=red] (B2) circle (4.0 pt);
\draw[fill=red] (B3) circle (4.0 pt);
\draw[fill=red] (B4) circle (4.0 pt);
%==============================
\draw[fill=red] (D1) circle (4.0 pt);
\draw[fill=red] (D2) circle (4.0 pt);
%\draw[fill=red] (D3) circle (4.0 pt);
\draw[fill=red] (D4) circle (4.0 pt);
\draw[fill=red] (D5) circle (4.0 pt);
%==============================
%
\node[rotate=0 ] at (A9) {$F$};
%\node[rotate=0 ] at (A10) {$K$};
\node[rotate=0 ] at (A11) {Skeleton mesh $\edgeH$};
% \draw[rotate=-45,>=stealth, color=black, very thin] (4.35,4.25) -- (4.35,2.25) ;
%\node[rotate=70] at (B5) {$d(x,z)$};
\draw[|<->|, xshift=0.1cm,yshift=-0.4cm] 
    ($ (A3) + (-.6,.9) $) -- ($ (A4) + (-.6,.9) $) node[fill=white, midway,sloped,scale=0.025cm] {$H$} ;
\end{tikzpicture}	
	};
	\node (D) at (2.8,-3.5) {
%
% With a minimal sub-mesh.
%
\begin{tikzpicture}[scale=.75]

%==========================
\coordinate (R1) at (-1.0,-1.50);
\coordinate (R2) at (-3,0);
\coordinate (R12) at (-1.5,-0.3);
\coordinate (R3) at (0.0,0.50);
%==========================
\draw[fill=gray!20,  very thick ] (R1) -- (R2) -- (R3) -- cycle ;
%==========================
\draw[fill=red] (R1) circle (2.0 pt);
\draw[fill=red] (R2) circle (2.0 pt);
\draw[fill=red] (R3) circle (2.0 pt);
%==========================
\node[rotate=0 ] at (-2.0,.8) {$\trih^{K_1}$};
\node[rotate=0 ] at (R12) {$\tau$};
%\node[rotate=0 ] at (3.8,-0.75) {$F$};
\draw[|<->|, xshift=0.1cm,yshift=-0.4cm] 
    ($ (R2) + (-.3,-.4) $) -- ($ (R1) + (-.3,-.4) $) node[fill=white, midway,sloped,scale=0.025cm] {$h$} ;
%==========================

%==========================
\coordinate (A7) at (5.0,0.50);
\coordinate (A8) at (3,-1.0);
\coordinate (A1) at (0.50,-1.50);
\coordinate (B3) at (1.50,0.50);
\coordinate (D1) at (5.0,0.50);
\coordinate (D2) at (3,-1.0);
\coordinate (D3) at (0.50,-1.50);
\coordinate (D4) at (1.50,0.50);
\coordinate (D5) at (1.75,-1.25);
\coordinate (D6) at (1,-0.5);
\coordinate (D7) at (2.25,-0.25);
\coordinate (D8) at (3.25,0.50);
\coordinate (D9) at (4.0,-.25);
\coordinate (D10) at (3.125,-.25);
\coordinate (D11) at (1.625,-0.375);

%==========================
\draw[fill=gray!20,  very thick ] (A7) -- (A8) -- (A1) -- (B3) -- cycle ;
\draw[fill=gray!20,  thin ]  (D7)  -- (2.75,0.50) ;
\draw[fill=gray!20,  thin ]  (D7)  -- (D9) ;
\draw[fill=gray!20,  thin ]  (D7)  -- (A8) ;
\draw[fill=gray!20,  thin ]  (D7)  -- (B3) ;
\draw[fill=gray!20,  thin ]  (D7)  -- (D6) ;
\draw[fill=gray!20,  thin ]  (D7)  -- (D5) ;
\draw[fill=gray!20,  thin ]  (3.75,0.50)  -- (D9) ;
\draw[fill=gray!20,  thin ]  (D5)  -- (D6) ;
\draw[fill=gray!20,  thin ]  (D10)  -- (A8) ;
\draw[fill=gray!20,  thin ]  (D10)  -- (2.75,0.50) ;
\draw[fill=gray!20,  thin ]  (D10)  -- (3.75,0.50) ;
%\draw[fill=gray!20,  thin ]  (D5)  -- (B3) ;
%\draw[fill=orange!20,  very thin ]  (A8)  -- (B3) -- cycle ;
%\draw[fill=gray!20,  very thin ]  (D5)  -- (D6) -- (D7) -- cycle ;
%\draw[fill=gray!20,  very thin ]  (D7)  -- (D8) -- (D9) -- cycle ;
%\draw[fill=gray!20,  very thin ]  (A8)  -- (D8) -- (D9) -- cycle ;
%

%==========================
%\draw[fill=gray] (A7) circle (4.0 pt);
%\draw[fill=gray] (A8) circle (4.0 pt);
%\draw[fill=gray] (A1) circle (4.0 pt);
%\draw[fill=gray] (B3) circle (4.0 pt);
%
\draw[fill=red] (D1) circle (2.0 pt);
\draw[fill=red] (D2) circle (2.0 pt);
\draw[fill=red] (D3) circle (2.0 pt);
\draw[fill=red] (D4) circle (2.0 pt);
\draw[fill=red] (D5) circle (2.0 pt);
\draw[fill=black] (D6) circle (2.0 pt);
\draw[fill=black] (D7) circle (2.0 pt);
\draw[fill=red] (D8) circle (2.0 pt);
\draw[fill=red] (D9) circle (2.0 pt);
\draw[fill=black] (D10) circle (2.0 pt);
%\draw[fill=black] (D11) circle (2.0 pt);
%
\draw[fill=black] (2.75,0.50) circle (2.0 pt);
\draw[fill=black] (3.75,0.50) circle (2.0 pt);
%
%==========================
\node[rotate=0 ] at (2.10,1.10) {$\trih^{K_2^{}}$};
\node[rotate=0 ] at (2.3,-0.9) {$\tau$};
%\node[rotate=0 ] at (3.8,-0.75) {$F$};
%==========================

\draw[dashed] (R1) -- (A1);
\draw[dashed] (R3) -- (B3);

\end{tikzpicture}	
	};
	\node (E) at (-1.5,-3) {Local meshes:};
% 	\draw[-> ,  thick] (3.2,-1.6) -- (1.8,-2.8);
% 	\draw[-> ,  thick] (4.8,-2) -- (4.8,-2.8);
\end{tikzpicture}
\caption{A two-dimensional polygon partitioned by the meshes $\calP$ and $\calT_h^{}$.
The fine-scale meshes, $\trih^{K_1}$ and $\trih^{K_2}$, are defined over $K_1, K_2 \in \calP$, respectively.
Elements $K_1$ and $K_2$ belong to $\calP$; faces $E$ and $F$ are in $\calE$ and $\calE_H^{}$, respectively; the simplexes $\tau$ belong to the affine mesh $\calT_h^{}$.}
\label{dom2d}
\end{figure}

On each $K \in \calP$, we define the discrete spaces
\begin{align*}
%\label{def:VhK}
\VV_h^{}(K) &:=\{ \vv_h^{}\in C^0(K)^d \,:\, \vv_h^{}\sVert[0]_{\tau}^{}\in\PP_k^{}(\tau)^d\,,\quad\forall\, \tau\in \calT_h^K\}\,,\\
%\label{def:tildeVhK}
\tildeVV_h^{}(K) &:= \VV_h(K) \cap \tildeVV^{}(K)\,,\\
%\label{Qh}
  Q_h(K) &:=\{ q_h\in C^0(K)\,:\, q_h\sVert[0]_{\tau}\in\PP_k(\tau)\,,\quad\forall\, \tau\in \calT_h^K\}\,,\\
%\label{bLHK}
  \bL_H(K) &:=\{ \bm_H\in \LL^2(\dK)\,:\, \bm_H\sVert[0]_{F}\in\PP_\ell(F)\,,\quad\forall\, F\in \calE_H^K\}\,.
\end{align*}
where $k,\ell \in \NN^+$ and $\PP_s(D)$ is the space of polynomials of degree less or equal to $s$ on $D$.
The corresponding global spaces are
\begin{equation*}
%\label{VhLh}
\begin{aligned}
\VV_h^{} &:= \{\vv_h\in \LL^2(\OO)\,:\, \, \vv_h\sVert[0]_K\in \VV_h(K)\,,\quad\forall\, K\in \calP\}\,,\\
\tildeVV_h^{} &:= \VV_h \cap \tildeVV^{}\,,\\
  Q_h&:= \{q_h\in L^2(\OO)\,:\, q_h\sVert[0]_K\in Q_h(K)\,,\quad\forall\, K\in \calP\}\,,\\
%\label{bLH}
\bL_H^{} &:= \{\bm_H\in \bL\,:\, \, \bm_H\sVert[0]_{\dK}\in \bL_H(K)\,,\quad\forall\, K\in \calP\}\,.
\end{aligned}
\end{equation*}
%As a general rule, we use the notation $$X(\calP) := \{v \in L^2(\OO) \,:\, v\sVert[0]_K \in X(K)\,,\quad\forall\, K\in \calP \}\,,$$ throughout all the paper.
%We introduce one extra notation $\upepsilon := (1-2\nu)/(2\,G\,\nu) \in L^\infty(\OO)$ to ease the presentation of the following results.
%Notice that $\upepsilon$ tends to zero when $\nu$ tends to $1/2$.

The following sections use Poincaré and Korn's inequalities on $K\in\calP$.
It is well-known (see \cite{ernGuermondFEM} for instance) that there exists a positive constant $C_{P,K}$ such that
\begin{align}
\label{eq:poincare}
\|\vv\|_{0,K} \leqslant C_{P,K}\,h_K^{}\,|\vv|_{1,K} \quad \text{for all }\, \vv \in \HH^1(K) \cap \LL^2_0(\OO) \subset \tildeVV(K)\,.
\end{align}
The Korn's inequality on the space $\tildeVV(K)$ was proved in \cite{10.1093/imanum/drac041}, i.e., there exists a positive constant $C_{korn,K}$, independent of $h_K$, such that
\begin{align}
\label{eq:korn2ndlocal}
|\tilde\vv|_{1,K} \leqslant C_{korn,K} \|\EE(\tilde\vv)\|_{0,K}\quad\text{for all }\tilde\vv \in \tildeVV(K)\,.
\end{align}
Finally, we use the following inverse inequality in the sequel
\begin{align}\label{eq:ineqCI}
C_I \sum_{\tau\in\calT_h^K} h_\tau^2 \left(\frac{1}{h_K^2}\|\svh\|_{0,\tau}^2 + \|\bdiv\,\svh\|_{0,\tau}^2\right)
\leqslant \|\svh\|_{0,K}^2\,,
\end{align}
for all $\vv_h \in \VV_h(K)$.
The constant $C_I$ depends only the polynomial order $k$ and $d$
% and on the shape of $K$ \todo{Verify this dependency} due to
using standard arguments, e.g., \cite[Lemma 1.138]{ernGuermondFEM}.
Hereafter, we use $C$, $C_1, C_2, \cdots$, for various positive constants which do not depend on $H$ or $h$, and do not degenerate when the Poisson ratio approaches $1/2$.

\begin{remark}
For star-shaped elements, the constant $C_{P,K}$ in \cref{eq:poincare} depends only on $d$ and the shape of $K$. See \cite{VeeVer12,di2017hybrid,ZheQui05} for different works on this topic.
\end{remark}

\subsection{The Galerkin Least Squares formulation}
\label{sec:GaLS}

Let $(T_h,T^p_h) : \bL \rightarrow \tildeVV_h\times Q_h$ and $(\hat T_h,\hat T^p_h) : \LL^2(\Omega) \rightarrow \tildeVV_h\times Q_h$ be linear operators defined by the following rule:
for each $\bm\in \bL$, $\bq\in \LL^2(\Omega)$, $(T_h,T^p_h)(\bm)$ and $(\hat T_h,\hat T^p_h)(\bq)$ satisfy
\begin{align}
\label{ThGaLS}
B_K(T_h(\bm), T^p_h(\bm); \tilde\vv_h, q_h) &= \langle \bm,\tilde\vv_h \rangle_{\dK}\,,\\ %&&\text{for all } (\tilde\vv_h, q_h) \in \tildeVV_h(K) \times Q_h(K)\,,\\
\label{hThGaLS}
B_K(\hat T_h(\bq), \hat T^p_h(\bq); \tilde\vv_h, q_h) &= F_{K}(\bq;\tilde\vv_h, q_h)\,, %&&\text{for all } \bq\in \LL^2(\Omega)\,,
\end{align}
for all $(\tilde\vv_h, q_h) \in \tildeVV_h(K) \times Q_h(K)$ and $K\in\triH$, where
\begin{align}
\label{eq:operatorBh}
B_K(\uu, p; \vv, q) &:= 
\int_K \left(2\,G\,\EE(\uu):\EE(\vv) - p\,(\Div\,\vv) - (\Div\,\uu)\,q - \upepsilon\, p\,q \right) \dif x\nonumber\\
& - \alpha_K\,\sum_{\tau\in\calT_h^K} h_\tau^2\, \int_\tau \left(\bdiv\,(2\,G\,\EE(\uu) - p\id_d) \,\cdot\, \bdiv\,(2\,G\,\EE(\vv) - q\id_d)\right) \dif x\,,\\
F_{K}(\bq;\vv, q)  &:= 
\int_K \bq\cdot\vv \dif x + \int_{\dK \cap \Gamma_N} \bg\cdot\vv \dif s\nonumber\\
& + \alpha_K\,\sum_{\tau\in\calT_h^K} h_\tau^2\, \int_\tau \left( \bq \,\cdot\, \bdiv\,(2\,G\,\EE(\vv) - q\id_d) \right) \dif x\,.
\label{eq:operatorFK}
\end{align}
%and $\alpha_K \in \RR$ is a stabilization parameter. %chosen carefully as shown in \cref{eq:alphaHyp}.
We choose the stabilization parameters $\alpha_K$ in the interval
\begin{align}\label{eq:alphaHyp}
0 < \alpha_K < \frac{G_{0,K}}{2\,\|G\|_{1,\infty,\calT_h^K}^2}\, C_I\,,
\end{align}
where $G_{0,K} := \ess\inf_{\xx\in K}G(\xx)$, and
\begin{align}\label{eq:normW1infty}
\|G\|_{1,\infty,\calT_h^K} := \max_{\tau \in \calT_h^K} \sqrt{\|G\|_{L^\infty(\tau)}^2 + h_K^2\|\nabla G\|_{L^\infty(\tau)}^2}\,,
\end{align}
Notice that the Least Squares terms in \cref{eq:operatorBh,eq:operatorFK} naturally induces the additional regularity over $G$, i.e., $G \in W^{1,\infty}(\calT_h) := \{v \in L^\infty(\OO) \,:\, v\sVert[0]_\tau \in W^{1,\infty}(\tau)\,,\quad\forall\, \tau\in \calT_h\}\,.$

The operators \cref{eq:operatorBh}-\cref{eq:operatorFK} are inspired in the Galerkin Least Squares (GaLS) method from \cite{stabilizedElasticity}.
The former work was proposed for elasticity problems with constant coefficients $G$ and $\nu$, and mixed Dirichlet-Neumann boundary.
Since the local problems \cref{e:pel}-\cref{e:pef} have pure Neumann boundary and heterogeneous elastic properties, we dedicate \cref{sec:wellPosednessTph} to adapt some results from \cite{stabilizedElasticity}.
%
%In \cref{sec:wellPosednessTph}, we prove the operators $(T_h,T^p_h)$ and $(\hat T_h,\hat T^p_h)$ are unique and stable, and that they approximate $\left(T,-\frac{1}{\upepsilon}\,\Div T\right)$ and $\left(\hat T,-\frac{1}{\upepsilon}\,\Div \hat T\right)$ with optimal rate,
%When $\alpha_K$ is chosen carefully (see \cref{eq:alphaHyp}).
%We shall mention the former version of the GaLS method was proposed to homogeneous Dirichlet elasticity problems, so that we must be careful when adapting it.
%
%We present the well-posedness of \cref{ThGaLS,hThGaLS} and error estimates with respect to \cref{e:pel,e:pef} in \cref{sec:wellPosednessTph}.
%
%In the same section, we show the operators $(T_h,T^p_h)$ and $(\hat T_h,\hat T^p_h)$ approximate $\left(T,-\frac{1}{\upepsilon}\,\Div T\right)$ and $\left(\hat T,-\frac{1}{\upepsilon}\,\Div \hat T\right)$.

~
\paragraph{Summary}

The MHM method with finite elements based on the GaLS formulation reads as follows:
Find the solution $(\bl_H,\uurm_\calH) \in \bL_H\times \VVRM^{}$ of \cref{hybrid-h}-\cref{hybrid-VVRM-h}, such that $(T_h,T^p_h)$ and $(\hat T_h,\hat T^p_h)$ satisfy \cref{ThGaLS,eq:operatorFK}.
This method provides the approximation for $\uu$ and $\stress\nn^K\sVert[0]_{\dK}$, equations \cref{decompositionUh,lambda}, and for $p$ and $\stress$ as follows
\begin{align}\label{uh-ph}
p_{Hh} &:=\, T_h^p(\bl_H) + \hat T_h^p\,(\ff)\,,\qquad
\stress_{Hh} := 2\,G\,\EE(\uu_{Hh})-p_{Hh}\,\id_d\,.
\end{align}
%In the following sections, we prove the method is well defined, and present some a priori estimates for the approximation errors.

\begin{remark}
We can use $\tilde\vv_h=\mathbf{0}$ and $q_h = 1$ in \cref{ThGaLS} and \cref{hThGaLS} to conclude the solution pair $(\uu_{Hh}, p_{Hh})$ satisfies the following local compressibility constraint
\begin{align}\label{rem:equilibrium}
\int_K \left(\Div\,\uu_{Hh} + \upepsilon\, p_{Hh}\right)\dif x = 0
 &&\text{for every } \,K\in \calP\,.
\end{align}
\end{remark}

\subsection{Well-posedness of the local formulation}
\label{sec:wellPosednessTph}

The following results adapts the proofs in \cite{stabilizedElasticity} to pure Neumann elasticity problems with heterogeneous coefficients $G$ and $\nu$.
Also, we introduce the $\upepsilon$-norm in $L^2(K)$ as
\begin{align}\label{epsNorm}
	\|q\|_{\upepsilon,K} := \left( \int_K (1+\upepsilon)\,q^2 \dif x \right)^{\frac12}\,,
	\quad \forall q \in L^2(K)\,,
\end{align}
and the $h$-seminorm in $H^1(\calT_h^K) := \{v \in L^2(K) \,:\, v\sVert[0]_\tau \in H^1(\tau)\,,\;\forall\, \tau\in \calT_h^K\}$ as
\begin{align}\label{hNorm}
	|q|_{h,K} := \left(\sum_{\tau\in\calT_h^K} h_\tau^2 \|\nabla q\|_{0,\tau}^2\right)^{\frac12}\,,
	\qquad \forall q \in H^1(\calT_h^K)\,.
\end{align}
We suppress using the sub-index $h$ to shorten formulas inside the proofs.

The next result addresses the continuity of $B_K$ \cref{hThGaLS} in $\tildeVV_h(K) \times Q_h(K)$.

\begin{lemma}[Boundness of $B_K$]\label{boundnessLemma}
There is a positive constant $C$ such that, for all $(\uu_h,p_h), (\vv_h,q_h)\in \tildeVV_h(K) \times Q_h(K)$, it holds
\begin{align}
	B_K(\uu_h,p_h;\vv_h,q_h) \leqslant C\, (\|\EE(\uu_h)\|_{0,K}^2+\|p_h\|_{\upepsilon,K}^2)^\frac12 (\|\EE(\vv_h)\|_{0,K}^2+\|q_h\|_{\upepsilon,K}^2)^\frac12.
\end{align}
\end{lemma}
\begin{proof}
We begin applying Cauchy-Schwarz inequalities in \cref{eq:operatorBh} to obtain
\begin{align*}
B_K(\uu_h,p_h;\vv_h,q_h) &\leqslant
%\|2G\|_{L^\infty(K)}\|\su\|_{0,K}\|\sv\|_{0,K}
%+\|p\|_{0,K}\|\Div\,\vv\|_{0,K}\\
%&\quad +\|\sqrt\upepsilon p\|_{0,K}\|\sqrt\upepsilon q\|_{0,K} + \|\Div\,\uu\|_{0,K}\|q\|_{0,K}\\
%&\quad + \alpha_K\sum_{\tau\in\calT_h^K} h_\tau^2 (\|\bdiv\,(2G\,\su-pI)\|_{0,\tau}\|\bdiv\,(2G\,\sv-qI)\|_{0,\tau})\\
%&\leqslant
N_K(\uu_h,p_h)\,N_K(\vv_h,q_h)\,,
\end{align*}
for $(\uu_h,p_h), (\vv_h,q_h)\in \tildeVV_h(K) \times Q_h(K)$,
where $N_K : \tildeVV_h(K) \times Q_h(K) \to \RR$ satisfies
\begin{multline*}
N_K(\uu_h,p_h)^2 := \|2G\|_{L^\infty(K)}\|\EE(\uu_h)\|_{0,K}^2
+\|p_h\|_{0,K}^2
+\|\sqrt\upepsilon p_h\|_{0,K}^2
+\|\Div\,\uu_h\|_{0,K}^2\\
+\alpha_K\sum_{\tau\in\calT_h^K} h_\tau^2\|\bdiv\,(2G\,\EE(\uu_h)-p_h\,I_d)\|_{0,\tau}^2\,.
\end{multline*}
Using \cref{th:auxiliarAlphaLemma2}, we obtain
$$
N_K(\uu_h,p_h)^2 \leqslant \left(\|2G\|_{L^\infty(K)} + d + 4 G_{0,K}\right)\|\EE(\uu_h)\|_{0,K}^2	
+ \|p_h\|_{\upepsilon,K}^2 + 4 G_{0,K} C_{0,K} \|p_h\|_{0,K}^2
$$
and the analogous expression for $N_K(\vv_h,q_h)^2$, so that the main result follows.
%Notice that if we consider $\alpha < C_K$:
%$$
%N_K(\uu,p)^2 <
%((1+2\,h_K)\,\|2G\|_{1,\infty,K}+1)\|\su\|_{0,K}^2
%+\left(1+\frac{1}{G_{0,K}}\right)\|p\|_{\upepsilon,K}^2,
%$$
%and then:
%$$C_M=1+\max\{9\,\|2G\|_{1,\infty,K},G_{0,K}^{-1}\}.$$
\end{proof}

%%\cite[Lemma 3.3]{stabilizedElasticity} gives an estimation
%%\begin{align}\label{preStabilityLemmaAux2}
%%\sup_{\mathbf 0 \neq \vv_h \in\VV_h(K)} \frac{(\Div\,\vv_h,p_h)_K}{\|\vv_h\|_{1,K}} \geqslant C_1 \|p_h\|_{0,K} - C_2 |p_h|_{h,K}\,,
%%\end{align}
%\begin{remark}\todo{Verify these dependencies!}
%The constants $C_1$ and $C_2$ depend only on (see the proof of \cite[Lemma 3.3]{stabilizedElasticity})
%\begin{enumerate}
%\item the constants on the Clément approximation, \cref{CleP1,CleP2};
%\item the parameter $\sigma_0$ and the shape of $K$ (see \cite[Eq.~2.5]{ste1990stokes} for more details on this).
%%, such that all $K \in \calP$ is the image of a continuous and injective mapping $G_{\tilde K} : \tilde K \rightarrow K$, where $\tilde K \in \tilde \calP$.
%%Such a dependence is associated to the constant in \cite[Eq.~2.5]{ste1990stokes}.
%%The other properties of $G_{\tilde K}$ are:
%%\begin{itemize}
%%\item $\tilde K = \cup_{\tau \in \calT_h^K} \tilde\tau$, where $\tilde\tau = G(\tau)$ are simplices disjoint two to two except for their boundaries;
%%\item $G{\sVert[0]_\tilde \tau} = F_{\tau} \circ F_{\tilde \tau}^{-1}$, $\tau \in \calT_h^K$, where $F_{\tau}$ and $F_{\tilde \tau}$ are the affine mappings from the reference element $\hat \tau$ onto $\tau$ and $\tilde \tau$ respectively.
%%\end{itemize}
%\end{enumerate}
%%
%\end{remark}

Using \cref{preStabilityLemma}, we prove the stability of $B_K$.

\begin{lemma}\label{stabilityLemma}
There is a positive constant $C$ such that, for all $(\tilde\uu_h,p_h)\in \tildeVV_h(K)\times Q_h(K)$, there exists a pair $(\tilde\vv_h,q_h) \in\tildeVV_h(K)\times Q_h(K)$ satisfying
\begin{align}
%\sup_{\substack{(\tilde\vv_h,q_h) \in\tildeVV_h(K)\times Q_h(K) \\ (\tilde\vv_h,q_h) \neq \mathbf 0}}
\frac{B_K(\tilde\uu_h,p_h;\tilde\vv_h,q_h)}{(\|\EE(\tilde\vv_h)\|_{0,K}^2+\|q_h\|_{\upepsilon,K}^2)^{1/2}} \geqslant C (\|\EE(\tilde\uu_h)\|_{0,K}^2+\|p_h\|_{\upepsilon,K}^2)^{1/2}\,.
\end{align}
%Moreover, there exists $(\tilde\vv_h,q_h)\in \tildeVV_h\times Q_h$ that achieves the supremum.
\end{lemma}
\begin{proof}
Let $\tilde\uu_h \in \tildeVV_h(K)$ and $p_h \in Q_h(K)$. Notice that
\begin{multline*}
B_K(\tilde\uu_h,p_h;\tilde\uu_h,-p_h) = \int_K (2G\,|\EE(\tilde\uu_h)|^2 + \upepsilon\, p_h^2)\dif x\\
- \alpha_K\,\sum_{\tau\in\calT_h^K} h_\tau^2\, (\|\bdiv\,(2\,G\,\EE(\tilde\uu_h))\|_{0,\tau}^2 - \|\nabla p_h\|_{0,\tau}^2 )\,.
\end{multline*}
From \cref{th:auxiliarAlphaLemma}, we obtain
\begin{align*}
\int_K 2G\,|\EE(\tilde\uu_h)|^2\dif x - \alpha_K\,\sum_{\tau\in\calT_h^K} h_\tau^2\, \|\bdiv\,(2\,G\,\EE(\tilde\uu_h))\|_{0,\tau}^2
%\\
\geqslant
%2 G_{0,K} \|\EE(\tilde\uu)\|_{0,K}^2 - \alpha_K C_I^{-1} \|2G\|_{1,\infty,K}^2 \|\EE(\tilde\uu)\|_{0,K}^2
%=
C_3 \|\EE(\tilde\uu_h)\|_{0,K}^2\,,
\end{align*}
where $C_3 := (2 G_{0,K} - \alpha_K C_I^{-1}) > 0$ using \cref{eq:alphaHyp}. Combine the two expressions above to obtain
\begin{align}\label{stabilityLemma-aux1}
B_K(\tilde\uu_h,p_h;\tilde\uu_h,-p_h) &\geqslant
C_3 \|\EE(\tilde\uu_h)\|_{0,K}^2
+\int_K \upepsilon\, p_h^2 \dif x
+\alpha_K |p_h|_{h,K}^2\,,
\end{align}

Next, let $\tilde\ww\in\tildeVV_h(K)$ be a function for which the supremum of \cref{preStabilityLemma} holds and such that
%This is possible by the proof of \cite[Lemma 3.3]{stabilizedElasticity}.
$\|\tilde\ww\|_{1,K} = \|p_h\|_{0,K}$.
Then, using the bilinearity of $B_K$, \cref{boundnessLemma,preStabilityLemma}, and the Cauchy-Schwarz inequality, we get
\begin{align*}
&B_K(\tilde\uu_h,p_h;-\tilde\ww,0)
= B_K(\tilde\uu_h,0;-\tilde\ww,0) + B_K(0,p_h;-\tilde\ww,0)\\
&\geqslant - C_4 \|\EE(\tilde\uu_h)\|_{0,K} \|\tilde\ww\|_{1,K} + \int_K p_h\,(\Div\,\tilde\ww) \dif x - \alpha_K \sum_{\tau\in\calT_h^K} h_\tau^2\, (\bdiv\,(2\,G\,\EE(\tilde\ww)), \nabla p_h)_\tau\\
&\geqslant - C_4 \|\EE(\tilde\uu_h)\|_{0,K} \|p_h\|_{0,K} + C_1 \|p_h\|_{0,K}^2 - C_2 |p_h|_{h,K} \|p_h\|_{0,K} \\
&\qquad - \sqrt{\alpha_K} \left(\alpha_K\sum_{\tau\in\calT_h^K} h_\tau^2\, \|\bdiv\,(2\,G\,\EE(\tilde\ww))\|_{0,\tau}^2 \right)^{1/2} \left(\sum_{\tau\in\calT_h^K} h_\tau^2\, \|\nabla p_h\|_{0,\tau}^2 \right)^{1/2}\,.
\end{align*}
Now we use \cref{th:auxiliarAlphaLemma}, \cref{eq:alphaHyp}, $\|\tilde\ww\|_{1,K} = \|p_h\|_{0,K}$ and the Cauchy-Schwarz inequality, to obtain
\begin{align}\label{stabilityLemma-aux2}
B_K(\tilde\uu_h,p_h;-\tilde\ww,0)
\geqslant
- C_6 \|\EE(\tilde\uu_h)\|_{0,K}^2 +  C_7 \,\int_K p_h^2 \dif x  - C_8 |p_h|_{h,K}^2\,,
\end{align}
where $C_6 := \frac{C_4}{2\gamma_1}$, $C_6 := C_1 - \frac{C_4\,\gamma_1}{2} - \frac{(C_2 + C_5) \gamma_2}{2}$ and $C_8 := \frac{C_2 + C_5}{2\gamma_2}$, and $\gamma_1$ and $\gamma_2$ are arbitrary positive constants.

Combining \cref{stabilityLemma-aux1,stabilityLemma-aux2}, we can define a constant $C_9$ such that
\begin{align}\label{stabilityLemma-aux3}
B_K(\tilde\uu_h,p_h;\tilde\uu_h-\delta\tilde\ww,-p_h)
%&= B_K(\tilde\uu_h,p_h;\tilde\uu_h,-p_h) + \delta B_K(\tilde\uu_h,p_h;-\tilde\ww,0)\\
%&\geqslant \left(C_3- \delta C_6\right) \|\EE(\tilde\uu)\|_{0,K}^2 + \int_K (\upepsilon + \delta C_7)\, p^2 \dif x + \left( \alpha_K - \delta C_8 \right) |p|_{h,K}^2\\
&\geqslant C_9 (\|\EE(\tilde\uu_h)\|_{0,K}^2 + \|p_h\|_{\upepsilon,K}^2)\,,
\end{align}
where $0<\delta < \min\left\{\frac{C_3}{C_6}, \frac{\alpha_K}{C_8}\right\}$.
On the other hand, choosing $\delta^2 \leqslant \frac12$, we obtain
\begin{multline}\label{stabilityLemma-aux4}
\|\EE(\tilde\uu_h -\delta \tilde\ww)\|_{0,K}^2 + \|-p_h\|_{\upepsilon,K}^2 \leqslant 2 \|\EE(\tilde\uu_h)\|_{0,K}^2 + 2 \delta^2 \|\EE(\tilde\ww)\|_{0,K}^2 + \|p_h\|_{\upepsilon,K}^2 =\\
= 2 \|\EE(\tilde\uu_h)\|_{0,K}^2 + \int_K (1 + \upepsilon + 2 \delta^2) p_h^2 \dif x \leqslant 2 (\|\EE(\tilde\uu_h)\|_{0,K}^2 + \|p_h\|_{\upepsilon,K}^2)\,,
\end{multline}
The main result follows from \cref{stabilityLemma-aux3,stabilityLemma-aux4} using $\vv_h=\tilde\uu_h-\delta\tilde\ww$, $q_h = -p_h$ and $C = C_9/\sqrt{2}$.
%, for all $(\tilde\uu,p)\in \tildeVV_h(K)\times Q_h(K)$ we have
%\begin{align*}
%\frac{B_K(\tilde\uu,p;\tilde\uu-\delta\tilde\ww,-p)}{(\|\EE(\tilde\uu -\delta \tilde\ww)\|_{0,K}^2 + \|-p\|_{\upepsilon,K}^2)^{1/2}} &\geqslant \frac{C_9 (\|\EE(\tilde\uu)\|_{0,K}^2 + \|p\|_{\upepsilon,K}^2)}{\sqrt 2 (\|\EE(\tilde\uu)\|_{0,K}^2 + \|p\|_{\upepsilon,K}^2)^{1/2}}\\
% &\geqslant C (\|\EE(\tilde\uu)\|_{0,K}^2 + \|p\|_{\upepsilon,K}^2)^{1/2}.
%\end{align*}
\end{proof}

The boundness of $F_{K}$ can be proved using the Cauchy-Schwarz and  triangle inequalities, and \cref{th:auxiliarAlphaLemma2}.
%\begin{align*}
%F_{K}(\bq;\vv, q)
%&= \sum_{\tau\in\calT_h^K} \int_\tau \bq\cdot\left( \vv + \alpha_K\,h_\tau^2\, \bdiv\,(2\,G\,\EE(\vv) - q\id_d) \right) \dif x\\
%&\leqslant \|\bq\|_{0,K} \left( \sum_{\tau\in\calT_h^K} \left\|\vv + \alpha_K\,h_\tau^2\, \bdiv\,(2\,G\,\EE(\vv) - q\id_d) \right\|_{0,\tau}^2 \right)^{\frac12}\\
%&\leqslant 2\,\|\bq\|_{0,K} \left( \left\|\vv\right\|_{0,K}^2 + \alpha_K^2\,\sum_{\tau\in\calT_h^K} h_\tau^4\, \left\|\bdiv\,(2\,G\,\EE(\vv) - q\id_d) \right\|_{0,\tau}^2 \right)^{\frac12}\\
%&\leqslant 2\,\|\bq\|_{0,K} \left( \left\|\vv\right\|_{0,K}^2 + 4\,G_{0,K}\,\alpha_K\,h^2\,\left(\|\sv\|_{0,K}^2 + C_{0,K} \|q\|_{0,K}^2\right) \right)^{\frac12}\\
%&\leqslant C \, \|\bq\|_{0,K} \, (\|\vv\|_{1,K}^2+\|q\|_{0,K}^2)^\frac12\,.
%\end{align*}
Therefore,
\cref{boundnessLemma,stabilityLemma} suffice to guarantee the well-posedness of the problems \cref{ThGaLS,hThGaLS} using the Banach-Ne\v{c}as-Babu\v{s}ka (BNB) Theorem \cite[Theorem~2.6]{ernGuermondFEM}.
%In fact,  the missing  $\langle \bm,\tilde\vv_h \rangle_{\dK} \leqslant \|\bm\|_{-1/2,\dK}\,\|\vv_h\|_{1,K}$
%Notice that the right-hand side of both equations are continuous.
%In fact, \cref{th:auxiliarAlphaLemma2} is used to prove the continuity of $F_{K}$ can be verified as follows
%
%We formalize the conclusions above in the next result.
%The next result tackles the stability of the upscaling operators $T_h$, $T^p_h$, $\hat T_h$ and $\hat T^p_h$.
%
As a direct result, the pairs $(T_h,T^p_h)$ and $(\hat T_h,\hat T^p_h)$ are well-defined.
Moreover, using Poincaré and Korn's inequalities \crefrange{eq:poincare}{eq:korn2ndlocal}, \cref{stabilityLemma-aux3}, \cref{ThGaLS}, and \cref{eq:bnd_b}, we obtain
\begin{align*}%\label{th:stabilityLemmaTh_aux1}
\|T_h(\bm)\|_{1,\calP}^2+\|T^p_h(\bm)\|_{0,\OO}^2
%&\leq \sum_{K \in \triH} C_{korn,K}^2\,(h_K^2\,C_{P,K}^2+1)\, \|\EE(T_h(\bm))\|_{0,K}^2 + \|T^p_h(\bm)\|_{\upepsilon,K}^2 \nonumber\\
%&\leq C\, \sum_{K \in \triH} B_K(T_h(\bm),T^p_h(\bm);T_h(\bm)-\delta\tilde\ww,-T^p_h(\bm)) \nonumber\\
%&= C\, \sum_{K \in \triH} \langle \bm, T_h(\bm)-\delta\tilde\ww\rangle_{\dK} \nonumber\\
&\le C\, \|\bm\|_{\bL} \, \|T_h(\bm)-\delta\tilde\ww\|_{1,\calP}\\
&\le 2\,C\, \|\bm\|_{\bL} \, \left( \|T_h(\bm)\|_{1,\calP}^2 + \|T^p_h(\bm)\|_{0,\calP}^2 \right)^\frac12\,.
\end{align*}
choosing $\delta$ small enough.
We proceed the same way for \cref{hThGaLS}, together with a trace inequality, e.g., \cite[Lemma~1.49]{di2011mathematical},
to conclude that there exists a positive constant $C$ satisfying
\begin{align}
\label{eq:stabilityLemmaTh}
\sqrt{\|T_h(\bm)\|_{1,\calP}^2+\|T^p_h(\bm)\|_{0,\OO}^2}
&\leq C \|\bm\|_{\bL}\,, \qquad \forall \bm \in \bL\,,\\
\label{eq:stabilityLemmahTh}
\sqrt{\|\hat T_h(\bq)\|_{1,\calP}^2+\|\hat T^p_h(\bq)\|_{0,\OO}^2}
&\leq C \left(\|\bq\|_{0,\OO} + \|\bg\|_{0,\Gamma_N}\right)\,, \qquad \forall \bq \in \LL^2(\OO)\,.
\end{align}

\subsection{Well-Posedness of the MHM method}
\label{sec:wellPosednessMHM}

The well-posedness of the MHM method on polytopal partitions was previously discussed in \cite{BarJaiParVal20,10.1093/imanum/drac041}.
The main ingredient to achieve well-posedness is to prove $T_h$ is injective on $\calN_H^{}$,
\begin{equation}
\label{ZH}
\calN_H^{}:=\left\{\bm_H^{}\in\bL_H^{}: \sum_{K\in\calP} \langle \bm_H, \vvrm\rangle_{\dK} = 0\,, \quad\forall\vvrm\in \VVRM\right\}\,.
\end{equation}
Since the local problems \cref{ThGaLS} are well-posed, $T_h$ is injective on $\calN_H^{}$ if and only if the following statement holds: (see the proof of \cite[Lemma~6.1]{Harder2016})
\begin{align}
\label{eq:injectivityOfTh}
	\bm_H \in \calN_H \;:\quad
	\sum_{K\in\calP} \langle \bm_H, \tilde\vv_h\rangle_{\dK} = 0\,, \quad\forall\tilde\vv_h\in \tildeVV_h
	\quad\Rightarrow\quad
	\bm_H = \mathbf 0\,.
\end{align}
Notice that, apart from $\tildeVV_h$, this condition has no relation with the local discrete scheme of the MHM method.

In \cite{Harder2016}, the authors prove the injectivity of $T_h$ on two-dimensional problems using
\begin{enumerate}
\item $\calT_h^{} = \calP$, and even polynomial degree $\ell$ under the constraint $k=\ell+1$;
%~\\
\item quadrilateral meshes $\calT_h^{} = \calP$ under  the constraint $k\geq\ell+2$.
\end{enumerate}
This result was extended in \cite{10.1093/imanum/drac041} for the case of two-level meshes $\calT_h$ that \textit{match} $\calE_H$, which means that each face $\mathrm f$ of an arbitrary $\tau \in \calT_h$ is contained by at most a single $F\in \calE_H$.
Under the matching condition and assuming $k-d\geq \ell \geq 1$, there exists a Fortin operator $\Pi_h^{}:\HH^1(\calP)\to \VV_h$, i.e., an operator that satisfies
\begin{equation}
\label{FortinLowk}
\begin{aligned}
\int_F\Pi_h^{}(\vv)\cdot\bm_H^{} \dif x &= \int_F \vv\cdot\bm_H^{} \dif x \quad\text{for all }\bm_H^{}\in \bL_H^{}\quad\text{and }\,F\in\calE_H^{}\,, \\
\|\Pi_h^{}(\vv)\|_{1,\calP}^{} &\le C\,\|\vv\|_{1,\calP}^{}\,,
\end{aligned}
\end{equation}
for all $\vv\in \HH^1(\calP)$.
The existence of such a Fortin operator also guarantees the injectivity of $T_h$ on $\calN_H$ \cite[Theorem~4.4]{10.1093/imanum/drac041}.

In the following result, we introduce a new sufficient condition for $T_h$ to be injective on $\calN_H^{}$ also based on the existence of a Fortin operator.
We cover the cases $k=\ell+1$ and $k=\ell$ for the two-dimensional case.

\begin{lemma}\label{th:supVVhLowk} 
Let $\OO \subset \RR^2$. Suppose $\calT_h$ \textit{matches} $\calE_H$, and one of the following conditions hold on each $K \in \calP$:
\begin{enumerate}
%\item
%$k \geqslant \ell+2$;
\item
$k \geqslant \ell+1 \geqslant 2$ and
there is at least 1 node of $\calT_h^K$ on each edge $F \in \edgeH \cap \dK$;
\item
$k \geqslant \ell \geqslant s$ and
there is at least $(4-s)$ nodes of $\calT_h^K$ on the interior of each edge $F \in \edgeH \cap \dK$, where $s \in \{1,2,3\}$.
%\item
%$k \geqslant \ell = 2$ and
%there is at least \textbf{two} nodes of $\calT_h^K$ on each face $F \in \edgeH \cap \dK$;
%\item
%$k \geqslant \ell = 1$ and
%there is at least \textbf{three} nodes of $\calT_h^K$ on each face $F \in \edgeH \cap \dK$.
\end{enumerate}
There exists a mapping $\Pi_h^{}:\HH^1(\calP)\to \VV_h$ such that, for all $\vv\in \HH^1(\calP)$, \cref{FortinLowk} holds.
As a consequence, $T_h$ is injective on $\calN_H^{}$.
%\begin{equation}
%\label{FortinLowk}
%\begin{aligned}
%\int_F\Pi_h^{}(\vv)\cdot\bm_H^{} \dif x &= \int_F \vv\cdot\bm_H^{} \dif x \quad\text{for all }\bm_H^{}\in \bL_H^{}\quad\text{and }\,F\in\calE_H^{}\,, \\
%\|\Pi_h^{}(\vv)\|_{1,\calP}^{} &\le C\,\|\vv\|_{1,\calP}^{}\,. 
%\end{aligned}
%\end{equation}
%where $C$ does not depend on $h$, $H$ nor $\calH$.
\end{lemma}
\begin{proof}
%When $k\geqslant \ell+d$, the result follows directly from \cref{l:Fortin}.
%%
%For the remaining cases,
Define
$
\mathcal{X} = \{q \in H_0^1(0,1) \,:\, q\sVert[0]_{(0,t)} \in \PP_{k+1}(0,t) \;\text{and}\; q\sVert[0]_{(t,1)} \in \PP_{k+1}(t,1) \}\,,
$
where $t \in (0,1)$.
Using the arguments in the proof of \cite[Lemma~4]{BarJaiParVal20}, we conclude that all $\varphi \in \PP_k(0,1)$ satisfying $\int_0^1 \varphi\,q\,\dif x = 0$ for all $q \in \mathcal{X}$ are identically zero in $[0,1]$.
In the same sense, the arguments used to prove \cite[Lemma~5]{BarJaiParVal20} are still valid over partitions of $(0,1)$ that are not equally spaced.
This being noticed, define $\Pi_h(\vv)_i := \Pi_h^{BJPV}(v_i)$, for each $i = 1, \cdots, d$, where $\Pi_h^{BJPV}:H^1(\calP)\to C^0(\calP) \cap \PP_k(\calT_h)$ is the operator from \cite[Lemma~2]{BarJaiParVal20}.
Using the properties of $\Pi_h^{BJPV}$, we conclude $\Pi_h$ satisfies the first equation in \cref{FortinLowk}.
Moreover, there exists a positive constant $C_K$ such that
% using \cref{def:XiHRegular} and proceeding like in \cite{BarJaiParVal20}, we obtain
\begin{align}\label{th:supVVhLowk_proof} 
\|\Pi_h^{BJPV}(v)\|_{1,K} \leqslant C_K\, \|v\|_{1,K}, \quad \text{for all } v \in H^1(K)\,,
\end{align}
for all $K \in \calP$.
Thus, we use \cref{th:supVVhLowk_proof} to complete the proof as follows
\begin{align*}
\|\Pi_h^{}(\vv)\|_{1,\calP}^{2}
&= \sum_{K \in \calP} \sum_{i=1}^d \|\Pi_h^{BJPV}(v_i)\|_{1,K}^2
\leqslant C\, \sum_{K \in \calP} \sum_{i=1}^d \|v_i\|_{1,K}^2
\leqslant C\, \|\vv\|_{1,\calP}^2\,.
\end{align*}
\end{proof}

For the three-dimensional case, one may use similar arguments to conclude that the Fortin mapping exists if, on each $K \in \calP$,
\begin{enumerate}
%\item
%$k \geqslant \ell+2$;
\item
$k = \ell = 1$ and
there is at least three non-collinear nodes of $\calT_h^K$ on the interior of face $F \in \edgeH \cap \dK$;
\item
$k = \ell+1 = 2$ and
there is at least one node of $\calT_h^K$ on the interior of each edge $E \subset \partial F$, where $F \in \edgeH \cap \dK$ is a face.
%\item
%$k \geqslant \ell = 2$ and
%there is at least \textbf{two} nodes of $\calT_h^K$ on each face $F \in \edgeH \cap \dK$;
%\item
%$k \geqslant \ell = 1$ and
%there is at least \textbf{three} nodes of $\calT_h^K$ on each face $F \in \edgeH \cap \dK$.
\end{enumerate}
%\end{lemma}
% We let the rigorous proof for future work.
%ideas can be generalized to three dimensional domains, although this is out of the scope of the current work.

%
%\begin{remark}\label{rem:MHMGalowOrder}
%It is important to highlight that, on two-dimensional domains, all results from \cref{elasticity} remain valid when using 
%%the constraint $k \geqslant \ell + d$ by any of
%the constraints from \cref{th:supVVhLowk} instead of $k=\ell + d$.
%%That is because the relation between polynomial exists only exists to guarantee the existence of a Fortin operator in the shape of \cref{FortinLowk}.
%As such, \cref{th:supVVhLowk} extends the family of pairs of interpolation spaces available to be used within the MHM method in two-dimensional problems.
%%This increases the range of configurations for the MHM from \cref{elasticity}.
%\end{remark}

%\begin{lemma}\label{th:supVVhLowk}
%Assume that $\ell \geq 0$ and $k \geqslant \ell$ satisfy
%\begin{align}
%???
%\end{align}
%\todo{Put the hypothesis}
%Therefore, there exists $C > 0$ such that, for every $\bm_H\in\, \calN_H$, satisfies
%\begin{align}
%\sup_{\vv\in \VV}\frac{\langle\bm_H,\vv\rangle_{\btriH}}{\|\vv\|_{\VV}}
%\leq C \sup_{\vv_h\in   \VV_h}\frac{\langle\bm_H,\vv_h\rangle_{\btriH}}{\| \vv_h\|_{\VV}}\,.
%\end{align}
%\end{lemma}
%%
%\begin{proof}
%\end{proof}

Provided $T_h$ is injective, one may prove the well-posedness of the MHM method as follows.
\begin{theorem}\label{stabilityGaLS}
Suppose $T_h$ is injective on $\calN_H^{}$.
Then, there exist positive constants $\alpha_0$ and $\beta_0^{}$, independent of $H$ and $h$ and that do not degenerate when the Poisson's ratio approaches $1/2$, such that
\begin{align}
\label{elipt-hGaLS}
\sum_{K\in\calP} \langle \bm_H, T_h(\bm_H)\rangle_{\dK} &\ge \alpha_0^{}\,\|\bm_H^{}\|_{\bL}^2 & \qquad & \text{for all $\bm_H\in \calN_H$}\,,\\
\label{inf-sup-hGaLS}
\sup_{\bm_H^{}\in\bL_H^{}\backslash\{\mathbf 0\}}\frac{\sum_{K\in\calP} \langle \bm_H, \vvrm\rangle_{\dK}}{\|\bm_H^{}\|_{\bL}^{}} 
&\ge \beta_0^{}\,\|\vvrm\|_{0,\OO}^{} & \qquad & \text{for all $\vvrm\in \VVRM$}\,.
\end{align}
Then, \cref{hybrid-h}-\cref{hybrid-VVRM-h} is well-posed.
\end{theorem}

\begin{proof}
We begin proving an auxiliary result.
Let $\tilde\ww\in\tildeVV_h$ be a function satisfying \cref{preStabilityLemma} and such that $\|\tilde\ww\|_{1,K} = \|T^p_h(\bm)\|_{0,K}$.
Therefore, for every $\delta > 0$, 
\begin{align*}
B_K(&T_h(\bm), T^p_h(\bm);T_h(\bm),0)\\
&= B_K\left(T_h(\bm), T^p_h(\bm);T_h(\bm) - \frac{\delta}{2} \tilde\ww,0\right) + \frac{\delta}{2} B_K(T_h(\bm), T^p_h(\bm); \tilde\ww,0)\\
&= B_K\left(T_h(\bm), T^p_h(\bm);T_h(\bm) - \frac{\delta}{2} \tilde\ww,-T^p_h(\bm)\right) + \frac{\delta}{2} B_K(-T_h(\bm), -T^p_h(\bm); - \tilde\ww,0)\\
&\geqslant \left(C_3- \delta C_6\right) \|\EE(T_h(\bm))\|_{0,K}^2 + \int_K (\upepsilon + \delta C_7)\,|T^p_h(\bm)|^2 \dif x +\\
&\qquad + \left( \alpha_K - \delta C_8 \right) \sum_{\tau \in \calT_h^K} h_\tau^2 |T^p_h(\bm)|_{1,\tau}^2\,,
\end{align*}
where the positive constants $C_3,C_6,C_7,C_8$ are defined in the proof of \cref{stabilityLemma}.
We can choose $\delta$ small enough to obtain a positive constant $C$ such that
\begin{align}\label{eq:stabilityBK}
B_K(T_h(\bm), T^p_h(\bm);T_h(\bm),0) \geqslant C (\|\EE(T_h(\bm))\|_{0,K}^2+\|T^p_h(\bm)\|_{0,K}^2)\,,
\end{align}
for all $\bm\in \bL$ and $K \in \triH$.
We use this formula to prove \cref{elipt-hGaLS}.

Let $\bm_H\in \calN_H$ and, for every $\vv\in\HH^1(\calP)$, denote $\tilde\vv= \vv-\Pi_{RM}^{}(\vv) \in \tildeVV$.
We can use \cref{eq:bnd_b}, \cref{ZH}, \cref{eq:stabilityVsum}, and $C_{korn} \geqslant 1$ to obtain
\begin{align}
\label{elipt-hGaLS-aux}
\|\bm_H\|_{\bL}
&\leq C_{korn}\, \sup_{\tilde\vv\in \tildeVV \backslash \{\mathbf 0\}}\frac{\sum_{K\in\triH}\langle\bm_H,\tilde\vv\rangle_{\dK}}{\| \tilde\vv\|_{1,\calP}}
\leq C\, \sup_{\tilde\vv_h\in \tildeVV_h \backslash \{\mathbf 0\}}\frac{\sum_{K\in\triH}\langle\bm_H,\tilde\vv_h\rangle_{\dK}}{\| \tilde\vv_h\|_{1,\calP}}\,.
\end{align}
The constant $C$ does not depend on $h$, $H$ or $\calH$ (see, for instance, \cite[Lemma~6.1]{Harder2016} and \cite[Lemma~10]{RavThom1977}).
\Cref{elipt-hGaLS} follows from \cref{elipt-hGaLS-aux}, \cref{ThGaLS}, \cref{boundnessLemma} and \cref{eq:stabilityBK}, in this order, with $q_h = \mathbf 0$, and finally \cref{ThGaLS} once more.
The inf-sup condition \cref{inf-sup-hGaLS} follows directly from \cite{10.1093/imanum/drac041}.
Equations \cref{elipt-hGaLS} and \cref{inf-sup-hGaLS} imply the stability and well-posedness of the formulation \cref{hybrid-h}-\cref{hybrid-VVRM-h} using the classical saddle-point theory.
\end{proof}

\begin{remark}
When there exists a Fortin operator $\Pi_h^{}:\HH^1(\calP)\to \VV_h$ satisfying \cref{FortinLowk}, the constant $C$ in \cref{elipt-hGaLS-aux} is bounded by $C_{korn}$, $C_P$ and $C$ from \cref{FortinLowk}.
The latter depends on the polynomial order $k$, $d$ and $\sigma_\calT$ for the case $k \geq \ell + 1$ as states \cite[Lemma 4.2]{10.1093/imanum/drac041}.
\end{remark}

%\begin{remark}
%For each $K\in\calP$, the finite element for the MHM method can be seen as the triplet $\left\{K,T_h(\bL_H)\sVert[0]_K,\Sigma_K\right\}$.
%The usual degrees of freedom in the MHM methods are of the form $\sigma^\dK \circ g_K^{-1} \in \Sigma_K$, where $g_K:\bL_H(K) \rightarrow T_h(\bL_H)\sVert[0]_K$ is a bijection satisfying $g_K\left(\bm_H\sVert[0]_\dK\right) = T_h(\bm_H)$ for all $\bm_H \in \bL_H$, and $\sigma^\dK$ is a classical degree of freedom associated to the discontinuous polynomial space $\bL_H(K)$.
%In \cref{sec:numResultsElasticity}, we use Lagrange finite elements for $\bL_H$, which means $\sigma^\dK$ are nodal values of functions in $\bL_H(K)$.
%The multiscale basis functions are $T_h(\bm_H)$, for all $\bm_H$ in the basis of $\bL_H$.
%The injectivity of $g_K$ follows from \cref{ThGaLS,eq:operatorBh,eq:stabilityBK,elipt-hGaLS-aux,FortinLowk}.
%\end{remark}

\section{A priori error estimates for the MHM method}
\label{sec:convergenceMHM}

The convergence estimates can be split into two parts: local, related to $T_h$ and $\hat T_h$ inside each element $K\in \calP$, and global, which is related to the best approximation of $\bL$ in $\bL_H$.
The next result tackles the approximation properties of the local formulation, i.e., it shows how good $(T_h,T^p_h)$ and $(\hat T_h,\hat T^p_h)$ approximate $\left(T,-\upepsilon^{-1}\,\Div T\right)$ and $\left(\hat T,-\upepsilon^{-1}\,\Div \hat T\right)$.
%Most deductions mimic \cite{stabilizedElasticity}.
%with respect to the variational equations \cref{e:pel} and \cref{e:pef}.
Moreover, it states the discrete formulation is locking-free.

\begin{theorem}\label{th:localErrorBound}
Let $\bm \in \bL$ and $\bq \in \LL^2(\OO)$ be such that $\tilde\uu := T(\bm)\sVert[0]_K + \hat T(\bq)\sVert[0]_K \in\HH^{s+1}(K)$ and $p := -\upepsilon^{-1}\Div\tilde\uu \in H^s(K)$, with $1 \leq s \leq k$.
The solutions $\tilde\uu_h := T_h(\bm)\sVert[0]_K+\hat T_h(\bq)\sVert[0]_K$ and $p_h := T^p_h(\bm)\sVert[0]_K+\hat T^p_h(\bq)\sVert[0]_K$, defined using \cref{ThGaLS}-\cref{hThGaLS}, satisfy the following estimate
\begin{align}\label{eq:localErrorBound}
\|\tilde\uu-\tilde\uu_h\|_{1,K} + \|p-p_h\|_{0,K}
\leqslant C\, h^s\,\left( |\tilde\uu|_{s+1,K} + |p|_{s,K}\right)\,.
\end{align}
\end{theorem}
\begin{proof}
Using the \cref{stabilityLemma}, the consistency of the schemas \cref{ThGaLS}-\cref{hThGaLS} for $\tilde\uu \in\HH^2(K)$ and $\upepsilon\,p \in H^1(K)$, and the proof of \cref{boundnessLemma}, we obtain
\begin{align*}
&(\|\EE(\tilde\uu_h-\tilde\ww_h)\|_{0,K}^2+\|p_h-s_h\|_{\upepsilon,K}^2)^{1/2} \leqslant
C \Big(
	\|2G\|_{L^\infty(K)} \|\EE(\tilde\uu-\tilde\ww_h)\|_{0,K}^2 \nonumber\\
	&\qquad\qquad + \|p-s_h\|_{0,K}^2
	+ \|\sqrt{\upepsilon} (p-s_h)\|_{0,K}^2
	+ \|\Div(\tilde\uu-\tilde\ww_h)\|_{0,K}^2 \nonumber\\
	&\qquad\qquad + \alpha_K \sum_{\tau\in\calT_h^K} h_\tau^2 \|\bdiv\,(2G\,\EE(\tilde\uu-\tilde\ww_h)-(p-s_h)I)\|_{0,\tau}^2
\Big)^{\frac12}\,,
\end{align*}
for all $(\tilde\ww_h,s_h) \in \tildeVV(K) \times Q_h$.
We then proceed like in the proof of \cref{th:auxiliarAlphaLemma,th:auxiliarAlphaLemma2} and, using \cref{hNorm}, obtain
\begin{align}
&(\|\EE(\tilde\uu_h-\tilde\ww_h)\|_{0,K}^2+\|p_h-s_h\|_{\upepsilon,K}^2)^{1/2} \nonumber
\leq C \Big(
	\|\EE(\tilde\uu-\tilde\ww_h)\|_{0,K}^2 \\
	&\qquad\qquad + \|p-s_h\|_{0,K}^2 + |p-s_h|_{h,K}^2 + \sum_{\tau\in\calT_h^K} h_\tau^2 \|\bdiv\,\EE(\tilde\uu-\tilde\ww_h)\|_{0,\tau}^2
\Big)^{\frac12}\,,
\label{eq:localErrorBound_aux2}
\end{align}
%where $C$ does not depend on $h$ or $\upepsilon$, but depends on $\|G\|_{L^\infty(K)}$, $\ess\inf_{\xx\in K}G(\xx)$ and $\ess\inf_{\xx\in K}\nu_0(\xx)$, among other terms.
%
Now, we choose $\ww_h$ to be the Scott-Zhang interpolation \cite{scottZhang1990Interpolation} of $\tilde\uu$ onto $\VV_h(K)$ and define $\tilde\ww_h := \ww_h - \Pi_{RM}\ww_h$.
Therefore,
\begin{align}
\label{eq:interpolation1}
\|\EE(\tilde\uu-\tilde\ww_h)\|_{0,K}
&\leqslant
|\tilde\uu-\ww_h|_{1,K}
\leqslant
C\,h^s\,|\tilde\uu|_{s+1,K}\,,\\
%\end{align*}
%and
%\begin{align}
%\begin{aligned}
\sum_{\tau \in \calT_h^K} h_\tau^2 \|\bdiv\,\EE(\tilde\uu-\tilde\ww_h)\|_{0,\tau}^2
&\leqslant
d\,\sum_{\tau \in \calT_h^K} h_\tau^2 |\tilde\uu-\ww_h|_{2,\tau}^2
%\nonumber\\
%&\leqslant
%d\,\sum_{\tau \in \calT_h^K} h_\tau^{2m} h_\tau^{2(1-m)} |\tilde\uu-\tilde\ww_h|_{2,\tau}^2 \nonumber\\
\leqslant
C\,h^{2s} |\tilde\uu|_{s+1,K}^2\,.
\label{eq:interpolation2}
%\end{aligned}
\end{align}
%
%the Scott-Zhang interpolation \cite{scottZhang1990Interpolation}
Analogously, we use $s_h$ as the Scott-Zhang interpolation of $p$ onto $Q_h(K)$, and verify the existence of a positive constant $C$ such that
\begin{align}
\label{eq:interpolation3}
\|p-s_h\|_{0,K} + |p-s_h|_{h,K} &\leqslant C\,h^s\,|p|_{s,K}\,.
\end{align}
Now, we replace \cref{eq:interpolation1}-\cref{eq:interpolation3} into \cref{eq:localErrorBound_aux2} to obtain
\begin{align}
\label{eq:localErrorBound_aux3}
(\|\EE(\tilde\uu_h-\tilde\ww_h)\|_{0,K}^2+\|p_h-s_h\|_{\upepsilon,K}^2)^{1/2}
&\leqslant
C\, h^s\,\left( |\tilde\uu|_{s+1,K}^2 + |p|_{s,K}^2\right)^{\frac12}\,.
\end{align}
Starting from the left side of \cref{eq:localErrorBound} we use
the Poincaré and Korn's inequalities \cref{eq:poincare,eq:korn2ndlocal},
sum and subtract $\tilde\ww_h$ and $\tilde s_h$,
use the triangle inequality,
and \cref{eq:interpolation1,eq:interpolation3,eq:localErrorBound_aux3},
to complete the proof.
\end{proof}
%
%\begin{corollary}
%\end{corollary}

As for the global part of the method, we use the best approximation result in $\bL_H$ from \cite{10.1093/imanum/drac041} to prove the following result.

\begin{theorem}
\label{converGaLS}
Suppose $T_h$ is injective on $\calN_H^{}$, and that $\alpha_K$ satisfies \cref{eq:alphaHyp} for all $K \in \calP$.
Moreover, suppose $\ell \geqslant 1$, and that there exists a mapping $\Pi_h^{}:\HH^1(\calP)\to \VV_h$ satisfying \cref{FortinLowk} for all $\vv\in \HH^1(\calP)$.
Let $\uu$ be the solution of \cref{elliptic}, and let $1 \leqslant s \leqslant k$ and $1 \leqslant m \leqslant \min\{s,\ell+1\}$ be such that
$\uu \in \HH^{s+1}(\calP)$, $p = -\upepsilon^{-1}\Div\tilde\uu \in H^s(\calP)$ and $\stress\in H^m(\calP)^{d\times d}$. Then, there exists a positive constant $C$ such that
\begin{equation}\label{error-est-1GaLS}
\begin{aligned}
\|\uurm-\,&\uurm_\calH\|_{0,\OO} + \|\bl-\bl_H\|_{\bL}\\
&\leqslant C\,\left( h^{s}(|\uu-\uurm|_{s+1,\triH} + |p|_{s,\triH}) + H^m\,|\stress|_{m,\calP} \right)\,.
\end{aligned}
\end{equation}
In addition, $\uu_{Hh}:=\uurm_\calH + T_h(\bl_H) + \hat T_h(\ff)$, $p_{Hh} :=\, T_h^p(\bl_H) + \hat T_h^p\,(\ff)$ and $\stress_{Hh} := 2\,G\,\EE(\uu_{Hh})-p_{Hh}\,\id_d$ satisfy
\begin{equation}\label{error-est-2GaLS}
\begin{aligned}
\|\uu-\,&\uu_{Hh}\|_{1,\calP} + \|p-p_{Hh}\|_{0,\OO} + \|\stress-\stress_{Hh}\|_{0,\OO}\\
&\leqslant C\,\left( h^s\,(|\uu-\uurm|_{s+1,\triH} + |p|_{s,\triH}) + H^m\,|\stress|_{m,\calP} \right)\,.
\end{aligned}
\end{equation}
\end{theorem}
\begin{proof}
\Cref{error-est-1GaLS} follows directly from the proof of the approximation result in \cite{10.1093/imanum/drac041}, using \cref{th:localErrorBound} as the local error estimate.
To prove \cref{error-est-2GaLS}, we decompose the error $\uu-\uu_{Hh}$ as follows
\begin{align*}
	&\|\uu-\uu_{Hh}\|_{1,\calP}\\
	&\leqslant \|\uurm-\uurm_\calH\|_{0,\OO} + |\uurm-\uurm_\calH|_{1,\calP} + \|T(\bl)-T_h(\bl_H) + (\hat T - \hat T_h)(\ff)\|_{1,\calP}\,.
\end{align*}
First, we use
the triangular inequality,
\cref{eq:stabilityLemmaTh},
\cref{eq:localErrorBound} and
\cref{error-est-1GaLS} to obtain
\begin{align*}
	\|\uurm-\uurm_\calH\|_{0,\OO} &+ \|T_h(\bl-\bl_H) + (T-T_h)(\bl) + (\hat T - \hat T_h)(\ff)\|_{1,\calP}\\
	&\leqslant  C\,\left( h^{s}(|\uu-\uurm|_{s+1,\triH} + |p|_{s,\triH}) + H^m\,|\stress|_{m,\calP} \right)\,.
\end{align*}
%We still need to estimate $|\uurm-\uurm_\calH|_{1,\calP}$.
%
In \cite{10.1093/imanum/drac041}, we proved there exists a constant $C$, depending only on $\OO$, the shape-regularity of $\calP$, and the Poincaré and trace inequality constants, such that
\begin{align*}
	|\uurm-\uurm_\calH|_{1,\calP} \leqslant C\,\left( \|\uurm-\uurm_\calH\|_{0,\OO} + |T(\bl)-T_h(\bl_H) + (\hat T - \hat T_h)(\ff)|_{1,\calP} \right)\,.
\end{align*}
Thus, arguing as before, we conclude
\begin{align*}
	|\uurm-\uurm_\calH|_{1,\calP}
	&\leqslant  C\,\left( h^{s}(|\uu-\uurm|_{s+1,\triH} + |p|_{s,\triH}) + H^m\,|\stress|_{m,\calP} \right)\,.
\end{align*}
We collect the four last expressions to conclude the first part of \cref{error-est-2GaLS}.
To estimate $\|p-p_{Hh}\|_{0,\OO}$, we use \cref{th:localErrorBound} and sum up the contributions on each $K\in \calP$.
To estimate $\|\stress-\stress_{Hh}\|_{0,\OO}$, observe that
$$\|\stress-\stress_{Hh}\|_{0,\OO}
\le 2\,G\,(|\uu-\uu_{Hh}|_{1,\calP}+\|p-p_{Hh}\|_{0,\OO})\,.
$$
The result follows from the previous estimates, and the proof is complete.
\end{proof}

\begin{remark}
The constant $C$ in \cref{eq:localErrorBound} depends on $\ess\inf_{\xx\in K}\nu_0(\xx) \ge \nu_0$.
Since the locking behavior appears for high values of $\nu_0$, i.e., $\nu_0 \approx 1/2$, this does not change the conclusion that the presented local method is free of locking.
Moreover,
if the material property $\upepsilon$ is constant on each $K\in \calP$, one may eliminate this dependency following the proof of \cite[Theorem~3.1]{stabilizedElasticity}.
Since  $\upepsilon$ does not appear explicitly in tho global problem \cref{hybrid-h}-\cref{hybrid-VVRM-h}, and the stability and convergence constants for the formulations \cref{ThGaLS,hThGaLS} do not degenerate when $\upepsilon \to 0$, the MHM method is locking-free.
\end{remark}

\begin{remark}
\Cref{converGaLS} states the MHM method's $(H+h)$-convergence, and that means we improve accuracy by refining the face and local partitions $\calE_H$ and $\calT_h$.
In order to achieve the typical $\calH$-convergence, one must define a family of global partitions $\{\calP_\calH\}_{\calH > 0}{}$ of $\OO$ such that all constants involved in the analysis stays bounded when $\calH \rightarrow 0$.
We refer to \cite{10.1093/imanum/drac041} for a detailed discussion on this topic.
\end{remark}

\begin{remark}
Under local smoothing conditions, one may improve the $\LL^2$-conver-gence order of the displacement error in \cref{th:localErrorBound}, adapting the procedure from the proof of \cite[Theorem 3.1]{stabilizedElasticity}.
To improve global $\LL^2$-convergence order in \cref{converGaLS}, we proceed exactly as in \cite{10.1093/imanum/drac041}.
\end{remark}

\section{Numerical Results}
\label{sec:numResultsElasticity}

This section assesses the MHM method numerically on nearly incompressible materials. We use the following two-dimensional model problem proposed in~\cite{Brenner93}. Let $\OO := [0,1]^2$ be an isotropic elastic domain with shear modulus $G=1$, Poisson's ratio $\nu \in \{0.2, 0.3, 0.4, 0.49, 0.499, 0.4999, 0.49999\}$, and let $\uu = \{u_1,u_2\}$,
\begin{align*}
u_1(x,y) &= (\cos(2 \pi x) - 1) \sin(2 \pi y) + \frac{1-2\,\nu}{2}\,\sin(\pi x) \sin(\pi y)\,,\\
u_2(x,y) &= (1-\cos(2 \pi y)) \sin(2 \pi x) + \frac{1-2\,\nu}{2}\,\sin(\pi x) \sin(\pi y)\,,
\end{align*}
be the exact solution of \cref{elliptic} with $\Gamma_D = \dO$.
The pressure is $p := -\frac{2\,G\,\nu}{1-2\nu}\,\Div\uu = -2\,\pi\, \nu\, \sin(\pi(x + y))$, $\stress := \fStress(\uu)$, and $\ff = \{f_1,f_2\}$ is defined as follows
\begin{align*}
f_1(x,y) &= \pi^2 \left( 4\sin(2\pi y)(2\cos(2\pi x)-1) - \cos(\pi(x+y)) + (1-2\nu)\sin(\pi x)\sin(\pi y) \right),\\
f_2(x,y) &= \pi^2 \left( 4\sin(2\pi x)(1 - 2\cos(2\pi y)) - \cos(\pi(x+y)) + (1-2\nu)\sin(\pi x)\sin(\pi y) \right).
\end{align*}
%
% Mind that the parameter $\alpha = 0.001$ is sufficient to obtain well-posed formulations in all cases for the GaLS and MHM-GaLS methods.

In the following, we verify the theoretical results and show some numerical estimates.
% The results show how Poisson's ratios close to 1/2 influence the MHM-Ga, MHM-GaLS, stdGalerkin, and GaLS solutions.
% Also, they help verify the theoretical convergence orders for the MHM-GaLS method.
%
The underlying algorithm and the implementation aspects of the MHM method implemented for the test cases were presented in \cite{10.1093/imanum/drac041}.

\subsection{The Poisson locking phenomenon}
\label{subsec:locking}

In this section, we see how high Poisson's ratios influence the precision of the MHM-Ga, MHM-GaLS, stdGalerkin, and GaLS methods.
We use structured triangular partitions in every stage of all methods, and the MHM methods use trivial skeleton meshes $\calE = \calE_H$.
We use stdGalerkin and GaLS with polynomial orders $k=1, 2, 3$.
The MHM methods use $\ell = 1$ in the global level and local finite elements defined by the stdGalerkin and GaLS discretizations with $k=1, 2, 3$.

As throughout the whole paper, $\uu_{Hh}$, $p_{Hh}$ and $\stress_{Hh}$ satisfy \cref{decompositionUh,uh-ph}, where $(\bl_{Hh},\uurm)$ is the solution of \cref{hybrid-h}-\cref{hybrid-VVRM-h}.
In the following, $\uu_h$, $\ww_h$, $\ww_{Hh}$ are the displacement solutions obtained by the GaLS, stdGalerkin and MHM methods, respectively, and $p_h$ and $\stress_h = 2\,G\,(\EE(\uu_h) - p_h\id_2)$ are the pressure and stress solutions obtained from the MHM-GaLS method.
%Remember that $\fStress(\cdot)$ was defined in \cref{stressFunc}.

\Cref{fig:errorsVaryingEps} compares the solutions of the four methods for all considered values of $\nu$.
They use the same orders $k = 1,2,3$ and mesh sizes $h = 2^{k-4.5}$, from top to bottom.
The global partition in the MHM methods have size $\calH = 2^{-1.5}$.
For $k=1$, the MHM-Ga method loses accuracy when increasing the Poisson's ratio (top-left).
In all cases, the MHM-Ga improves the accuracy of the stdGalerkin method for a sufficiently high Poisson's ratio.
The MHM-GaLS method does not show a pronounced lack of accuracy in the approximation errors, evidencing its locking-free property.
However, the high Poisson's ratio affects more the MHM-GaLS than the GaLS method for all $k$.
As expected, the GaLS and MHM-GaLS improve the accuracy of the stdGalerkin and MHM-Ga methods, respectively, when the Poisson's ratio is close to $1/2$.

\begin{figure}[!htb]
\centering
\mincludegraphics{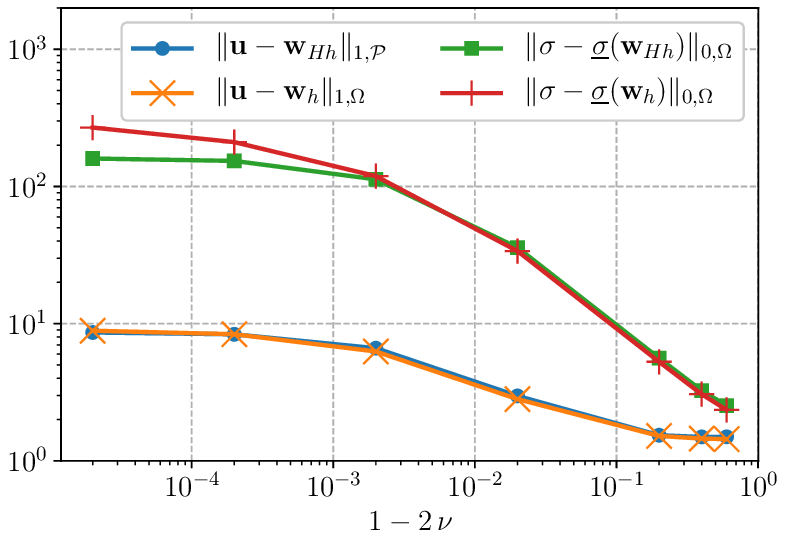}
\hfill
\mincludegraphics{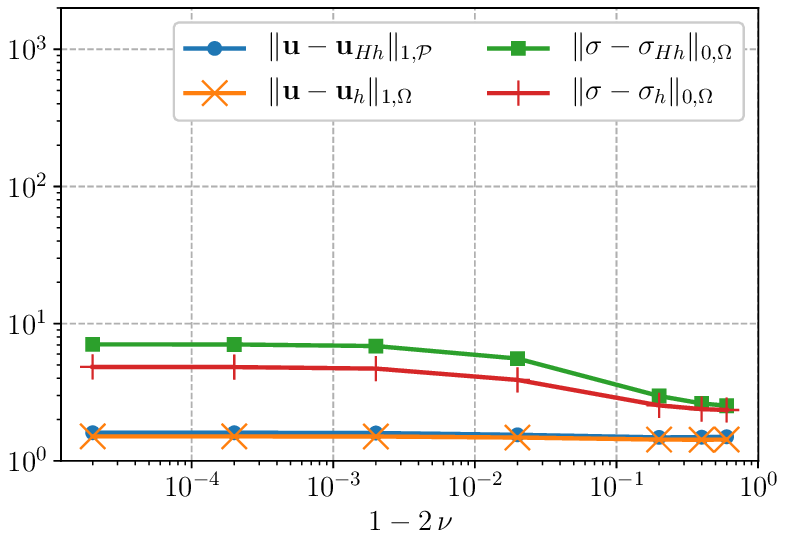}
\mincludegraphics{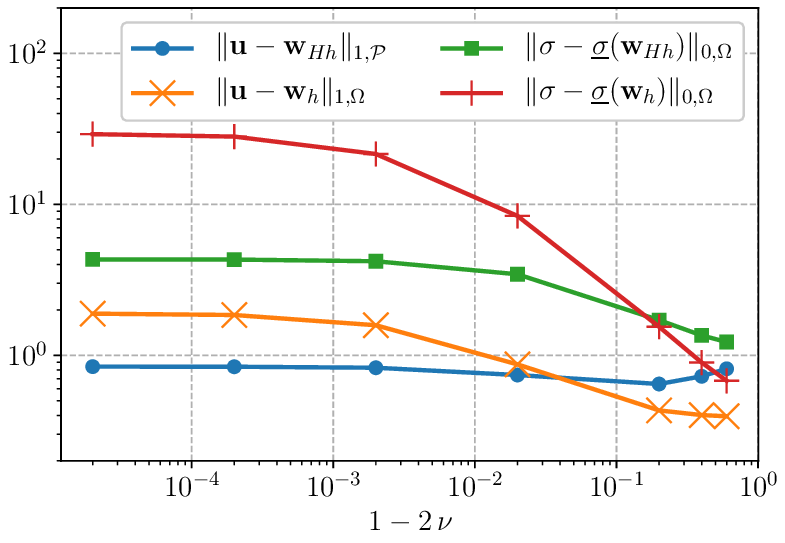}
\hfill
\mincludegraphics{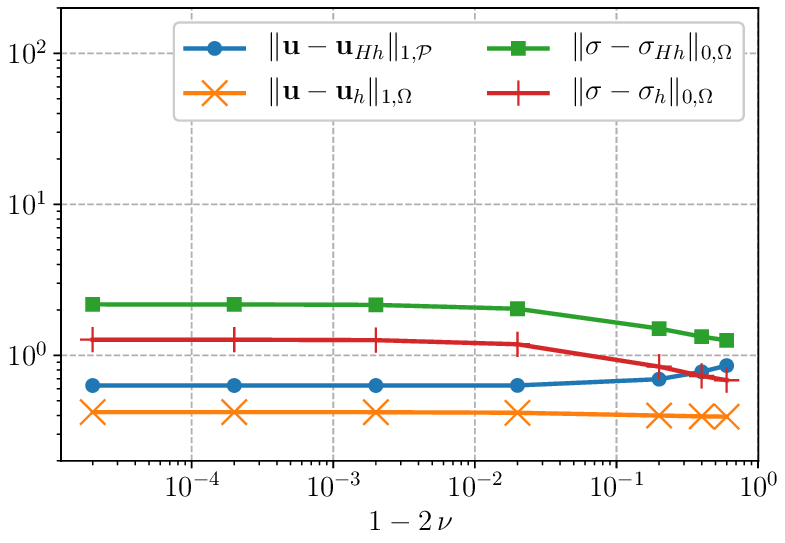}
\mincludegraphics{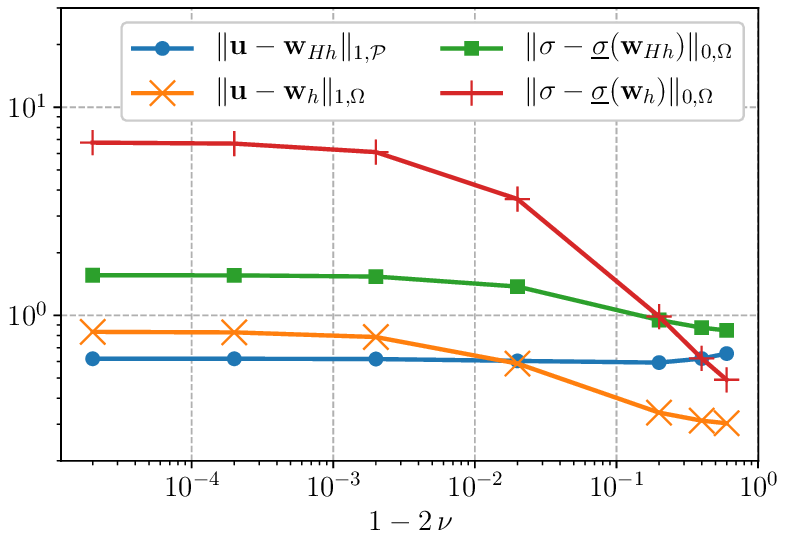}
\hfill
\mincludegraphics{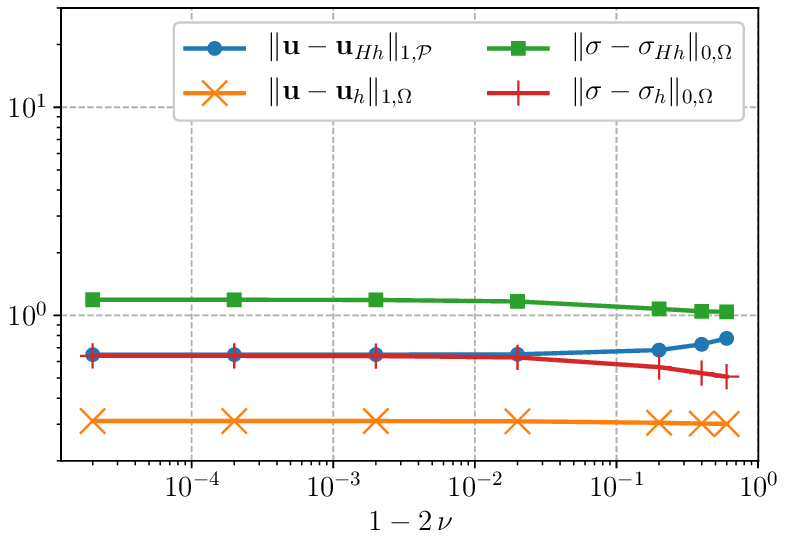}
\caption{$\upepsilon$-convergence in the four different methods.
Orders $k = 1,2,3$ and mesh sizes $h = 2^{k-4.5}$, from top to bottom.
The global partition's size in the MHM methods is $H = 2^{-1.5}$.}
\label{fig:errorsVaryingEps}
\end{figure}

% \begin{figure}[!htb]
% \centering
% \mincludegraphics{stdGalerkinvsMHMGa_convergenceh_k1}
% \hfill
% \mincludegraphics{stdGalerkinvsMHMGa_convergenceh_k2}
% \caption{$h$-convergence in the MHM-Ga and stdGalerkin methods using $\nu = 0.4999$, and orders $k = 1$ (left) and $2$ (right).
% The size of the global partition in the MHM-Ga is $\calH = 2^{3-k}\,h$. CORRIGIR LEGENDA NAS FIGURAS ??? }
% \label{fig:errorsVaryingh}
% \end{figure}

% \Cref{fig:errorsVaryingh} shows the $h$-convergence in the MHM-Ga and stdGalerkin methods using $\nu = 0.4999$, and $\calH = 2^{3-k}\,h$ for the MHM-Ga.
% They use the same polynomial orders $k = 1,2$.
% %
% The Poisson locking appears in the graph on the left ($k=1$), where the theoretical convergence order $O(h)$ is polluted for small $h$.
% Also, observe the stress error increase while reducing $h$ in the two methods.
% %
% For $k=2$, we see the the MHM-Ga method is more robust than the stdGalerkin method with respect to the locking.

\Cref{fig:errorsVaryinghGaLS} compares the solutions of the MHM-GaLS and GaLS methods using $\nu = 0.4999$, $\calH = 2^{3-k}\,h$ for the MHM-GaLS, and $k = 1,2$.
Both methods are locking-free, as theory predicts. 
However, some curves on both methods approach the theoretical $O(h^2)$ convergence very slowly.
We verified this is not related to a bad choice of parameter $\alpha$.

\begin{figure}[!htb]
\centering
\mincludegraphics{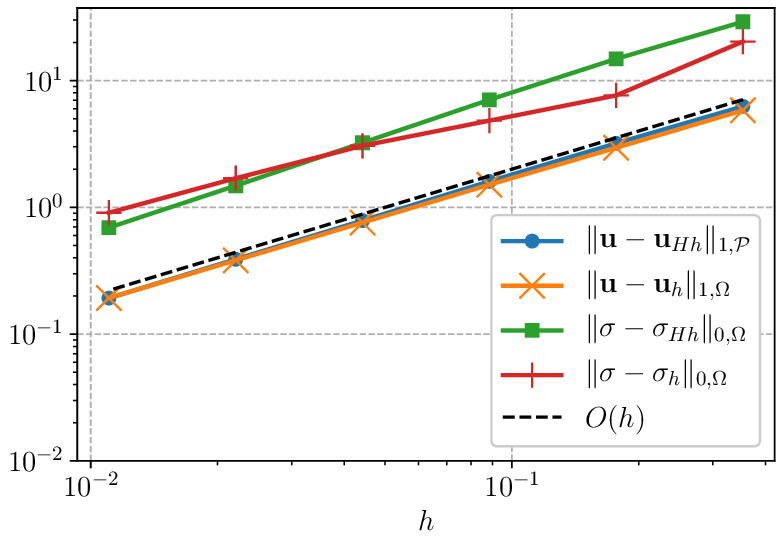}
\hfill
\mincludegraphics{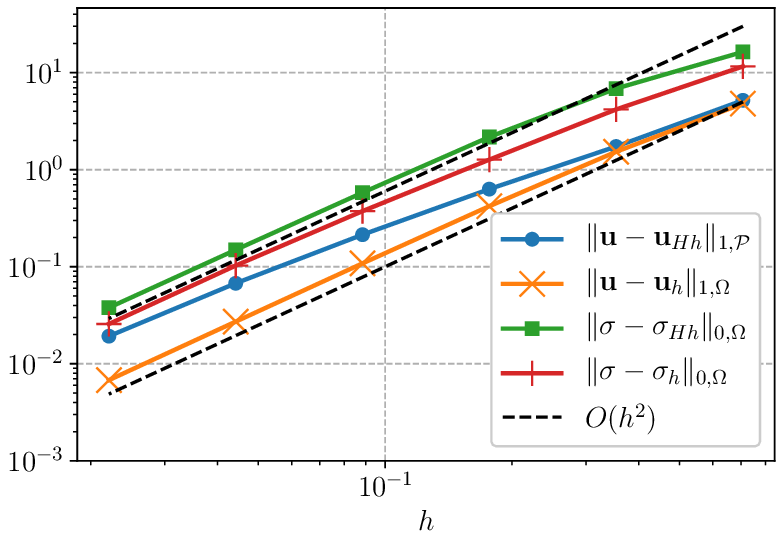}
\caption{$h$-convergence in the MHM-GaLS and GaLS methods using $\nu = 0.4999$, and
orders $k = 1$ (left) and $2$ (right).
The size of the global partition in the MHM-GaLS is $\calH = 2^{3-k}\,h$.}
\label{fig:errorsVaryinghGaLS}
\end{figure}

\subsection{$H$-convergence tests}
\label{subsec:analytical_cases}

In this section, we show numerical tests to verify the MHM-GaLS method's $H$-convergence rates, and compare the results with the MHM-Ga method.
We use a structured global partition with 32 triangles ($\calH = 2^{-1.5}$), and polynomial orders $\ell = 1$ and $k=1, 2, 3$.
We highlight that we study the convergence through a skeleton-based $H$-refinement, keeping $\calH$ fixed.
The second-level meshes' sizes are $h = 2^{k-3}\,H$, corresponding to the minimum allowed for well-posedness (see \cref{th:supVVhLowk}).
The results use the problem from \cref{subsec:locking} with $\nu = 0.4999$ and $k=1, 2, 3$.

\Crefrange{tab:HconvMHMGaLSk1}{tab:HconvMHMGak3} show the displacement, stress and pressure errors, and their respective numerical convergence orders.
We measure the errors $\mathbf e_{Hh} := \uu - \uu_{Hh}$, $\mathbf e_{Hh}^s := \stress - \stress_{Hh}$, $e_{Hh}^p := p - p_{Hh}$, $\mathbf g_{Hh} := \uu - \ww_{Hh}$, $\mathbf g_{Hh}^s := \stress - \fStress(\ww_{Hh})$ using definitions from \cref{subsec:locking}.
We employed a usual procedure to obtain these orders, namely, we divide the errors in one line by the errors in the line above, and apply the logarithm base 2.

\begin{table}[!htb]
\caption{$H$-convergence in the MHM-GaLS using $k=1$. `ord' is the numerical convergence order.}
\label{tab:HconvMHMGaLSk1}
\sisetup{
	round-mode=figures,
	round-precision=3,
	scientific-notation=true,
	exponent-product = \cdot}
\centering
\begin{small}
\begin{tabular}{c|c|c|c|c|c|c|c|c}
$H$ & $\|\mathbf e_{Hh}\|_{0,\Omega}$ & ord & $|\mathbf e_{Hh}|_{1,\mathcal{P}}$ & ord & $\|\mathbf e^s_{Hh}\|_{0,\Omega}$ & ord & $\| e^p_{Hh}\|_{0,\Omega}$ & ord \\\hline
$\calH/2^{0}$ & \num{5.048827e-02} & - & \num{1.582548e+00} & - & \num{4.769215e+00} & - & \num{2.742063e+00} & - \\
$\calH/2^{1}$ & \num{1.314001e-02} & 1.94 & \num{7.477040e-01} & 1.08 & \num{1.855571e+00} & 1.36 & \num{9.502105e-01} & 1.53 \\
$\calH/2^{2}$ & \num{3.218788e-03} & 2.03 & \num{3.664956e-01} & 1.03 & \num{8.475293e-01} & 1.13 & \num{4.067385e-01} & 1.22 \\
$\calH/2^{3}$ & \num{7.847502e-04} & 2.04 & \num{1.816250e-01} & 1.01 & \num{3.830576e-01} & 1.15 & \num{1.632677e-01} & 1.32 \\
$\calH/2^{4}$ & \num{1.929612e-04} & 2.02 & \num{9.046751e-02} & 1.01 & \num{1.738870e-01} & 1.14 & \num{6.090531e-02} & 1.42 \\
$\calH/2^{5}$ & \num{4.779859e-05} & 2.01 & \num{4.515416e-02} & 1.00 & \num{8.125988e-02} & 1.10 & \num{2.205976e-02} & 1.47
\end{tabular}
\end{small}
\sisetup{scientific-notation=false}
\end{table}

\begin{table}[!htb]
\caption{$H$-convergence in the MHM-GaLS using $k=2$. `ord' is the numerical convergence order.}
\label{tab:HconvMHMGaLSk2}
\sisetup{
	round-mode=figures,
	round-precision=3,
	scientific-notation=true,
	exponent-product = \cdot}
\centering
\begin{small}
\begin{tabular}{c|c|c|c|c|c|c|c|c}
$H$ & $\|\mathbf e_{Hh}\|_{0,\Omega}$ & ord & $|\mathbf e_{Hh}|_{1,\mathcal{P}}$ & ord & $\|\mathbf e^s_{Hh}\|_{0,\Omega}$ & ord & $\|e^p_{Hh}\|_{0,\Omega}$ & ord \\\hline
$\calH/2^{0}$ & \num{2.416143e-02} & - & \num{8.462417e-01} & - & \num{1.868701e+00} & - & \num{9.305267e-01} & - \\
$\calH/2^{1}$ & \num{2.162101e-03} & 3.48 & \num{1.708587e-01} & 2.31 & \num{4.351330e-01} & 2.10 & \num{2.233061e-01} & 2.06 \\
$\calH/2^{2}$ & \num{2.228503e-04} & 3.28 & \num{3.790887e-02} & 2.17 & \num{1.232338e-01} & 1.82 & \num{6.804818e-02} & 1.71 \\
$\calH/2^{3}$ & \num{2.276348e-05} & 3.29 & \num{8.332302e-03} & 2.19 & \num{3.517234e-02} & 1.81 & \num{2.041021e-02} & 1.74 \\
$\calH/2^{4}$ & \num{2.351036e-06} & 3.28 & \num{1.864617e-03} & 2.16 & \num{9.513409e-03} & 1.89 & \num{5.645561e-03} & 1.85 \\
$\calH/2^{5}$ & \num{2.540853e-07} & 3.21 & \num{4.327648e-04} & 2.11 & \num{2.472438e-03} & 1.94 & \num{1.481480e-03} & 1.93
\end{tabular}
\end{small}
\sisetup{scientific-notation=false}
\end{table}

\begin{table}[!htb]
\caption{$H$-convergence in the MHM-GaLS using $k=3$. `ord' is the numerical convergence order.}
\label{tab:HconvMHMGaLSk3}
\sisetup{
	round-mode=figures,
	round-precision=3,
	scientific-notation=true,
	exponent-product = \cdot}
\centering
\begin{small}
\begin{tabular}{c|c|c|c|c|c|c|c|c}
$H$ & $\|\mathbf e_{Hh}\|_{0,\Omega}$ & ord & $|\mathbf e_{Hh}|_{1,\mathcal{P}}$ & ord & $\|\mathbf e^s_{Hh}\|_{0,\Omega}$ & ord & $\|e^p_{Hh}\|_{0,\Omega}$ & ord \\\hline
$\calH/2^{0}$ & \num{2.777963e-02} & - & \num{8.990255e-01} & - & \num{1.571036e+00} & - & \num{7.085649e-01} & - \\
$\calH/2^{1}$ & \num{2.116110e-03} & 3.71 & \num{1.414725e-01} & 2.67 & \num{2.814937e-01} & 2.48 & \num{1.270090e-01} & 2.48 \\
$\calH/2^{2}$ & \num{1.933996e-04} & 3.45 & \num{2.645930e-02} & 2.42 & \num{5.561834e-02} & 2.34 & \num{2.538509e-02} & 2.32 \\
$\calH/2^{3}$ & \num{1.770981e-05} & 3.45 & \num{4.861420e-03} & 2.44 & \num{1.033594e-02} & 2.43 & \num{4.798035e-03} & 2.40 \\
$\calH/2^{4}$ & \num{1.583837e-06} & 3.48 & \num{8.701950e-04} & 2.48 & \num{1.856513e-03} & 2.48 & \num{8.680592e-04} & 2.47 \\
$\calH/2^{5}$ & \num{1.402840e-07} & 3.50 & \num{1.542378e-04} & 2.50 & \num{3.297074e-04} & 2.49 & \num{1.546014e-04} & 2.49
\end{tabular}
\end{small}
\sisetup{scientific-notation=false}
\end{table}

\begin{table}[!htb]
\caption{$H$-convergence in the MHM-Ga using $k=1$. `ord' is the numerical convergence order.}
\label{tab:HconvMHMGak1}
\sisetup{
	round-mode=figures,
	round-precision=3,
	scientific-notation=true,
	exponent-product = \cdot}
\centering
\begin{small}
\begin{tabular}{c|c|c|c|c|c|c}
$H$ & $\|\mathbf g_{Hh}\|_{0,\Omega}$ & ord & $|\mathbf g_{Hh}|_{1,\mathcal{P}}$ & ord & $\|\mathbf g^s_{Hh}\|_{0,\Omega}$ & ord \\\hline
$\calH/2^{0}$ & \num{8.757391e-01} & - & \num{7.285279e+00} & - & \num{1.237460e+02} & - \\
$\calH/2^{1}$ & \num{9.896605e-01} & -0.18 & \num{7.511295e+00} & -0.04 & \num{2.689118e+02} & -1.12 \\
$\calH/2^{2}$ & \num{7.880138e-01} & 0.33 & \num{5.915097e+00} & 0.34 & \num{3.931797e+02} & -0.55 \\
$\calH/2^{3}$ & \num{4.449403e-01} & 0.82 & \num{3.483533e+00} & 0.76 & \num{3.979385e+02} & -0.02 \\
$\calH/2^{4}$ & \num{1.802393e-01} & 1.30 & \num{1.551251e+00} & 1.17 & \num{2.923860e+02} & 0.44 \\
$\calH/2^{5}$ & \num{5.689783e-02} & 1.66 & \num{5.394682e-01} & 1.52 & \num{1.750894e+02} & 0.74
\end{tabular}
\end{small}
\sisetup{scientific-notation=false}
\end{table}

\begin{table}[!htb]
\caption{$H$-convergence in the MHM-Ga using $k=2$. `ord' is the numerical convergence order.}
\label{tab:HconvMHMGak2}
\sisetup{
	round-mode=figures,
	round-precision=3,
	scientific-notation=true,
	exponent-product = \cdot}
\centering
\begin{small}
\begin{tabular}{c|c|c|c|c|c|c}
$H$ & $\|\mathbf g_{Hh}\|_{0,\Omega}$ & ord & $|\mathbf g_{Hh}|_{1,\mathcal{P}}$ & ord & $\|\mathbf g^s_{Hh}\|_{0,\Omega}$ & ord \\\hline
$\calH/2^{0}$ & \num{2.683079e-02} & - & \num{8.632081e-01} & - & \num{3.029164e+00} & - \\
$\calH/2^{1}$ & \num{2.959724e-03} & 3.18 & \num{2.550048e-01} & 1.76 & \num{1.243074e+00} & 1.29 \\
$\calH/2^{2}$ & \num{5.216164e-04} & 2.50 & \num{1.011932e-01} & 1.33 & \num{9.344063e-01} & 0.41 \\
$\calH/2^{3}$ & \num{1.161420e-04} & 2.17 & \num{4.616516e-02} & 1.13 & \num{8.671651e-01} & 0.11 \\
$\calH/2^{4}$ & \num{2.578072e-05} & 2.17 & \num{2.063337e-02} & 1.16 & \num{7.697612e-01} & 0.17 \\
$\calH/2^{5}$ & \num{4.727642e-06} & 2.45 & \num{7.669352e-03} & 1.43 & \num{5.480134e-01} & 0.49
\end{tabular}
\end{small}
\sisetup{scientific-notation=false}
\end{table}

\begin{table}[!htb]
\caption{$H$-convergence in the MHM-Ga using $k=3$. `ord' is the numerical convergence order.}
\label{tab:HconvMHMGak3}
\sisetup{
	round-mode=figures,
	round-precision=3,
	scientific-notation=true,
	exponent-product = \cdot}
\centering
\begin{small}
\begin{tabular}{c|c|c|c|c|c|c}
$H$ & $\|\mathbf g_{Hh}\|_{0,\Omega}$ & ord & $|\mathbf g_{Hh}|_{1,\mathcal{P}}$ & ord & $\|\mathbf g^s_{Hh}\|_{0,\Omega}$ & ord \\\hline
$\calH/2^{0}$ & \num{2.502421e-02} & - & \num{8.055023e-01} & - & \num{1.435549e+00} & - \\
$\calH/2^{1}$ & \num{1.960199e-03} & 3.67 & \num{1.326745e-01} & 2.60 & \num{3.713559e-01} & 1.95 \\
$\calH/2^{2}$ & \num{2.067150e-04} & 3.25 & \num{2.732785e-02} & 2.28 & \num{1.723788e-01} & 1.11 \\
$\calH/2^{3}$ & \num{2.496859e-05} & 3.05 & \num{6.368240e-03} & 2.10 & \num{1.241695e-01} & 0.47 \\
$\calH/2^{4}$ & \num{2.807368e-06} & 3.15 & \num{1.392237e-03} & 2.19 & \num{6.976607e-02} & 0.83 \\
$\calH/2^{5}$ & \num{2.212950e-07} & 3.67 & \num{2.206648e-04} & 2.66 & \num{2.165042e-02} & 1.69
\end{tabular}
\end{small}
\sisetup{scientific-notation=false}
\end{table}

We observe no locking effect on \Cref{tab:HconvMHMGaLSk1,tab:HconvMHMGaLSk2,tab:HconvMHMGaLSk3}, as the convergence orders stay close to and, in some cases, above the predicted order (see \cref{converGaLS}).
In \cref{tab:HconvMHMGaLSk1,tab:HconvMHMGaLSk2}, we observe higher and lower convergence orders for $\|\mathbf e^p_{Hh}\|_{0,\Omega}$, respectively, when compared to the theory.
These fluctuations both influence the respective convergence order of $\|\mathbf e^s_{Hh}\|_{0,\Omega}$ since $\stress_{Hh}$ depends on $p_{Hh}$.
An additional $O(H^{0.5})$ convergence appears in \cref{tab:HconvMHMGaLSk3} for all measured errors, which is in accordance with what was found in other MHM methods, e.g., \cite{10.1093/imanum/drac041}.

\Cref{tab:HconvMHMGak1,tab:HconvMHMGak2,tab:HconvMHMGak3} shows the MHM-Ga methods' convergence orders that, theoretically \cite{10.1093/imanum/drac041}, should also satisfy the same estimate \cref{error-est-2GaLS}.
We observe a convergence losses for all orders $k$.
Since the MHM-GaLS method do not present such issues, and they are less evident in higher orders ($k=3$), we conclude the convergence losses for low $H$ are due to the Poisson's locking.
Finally, we shall highlight the MHM-GaLS recovered highly accurate discrete stress fields and, on nearly-incompressible materials, it improves considerably the MHM-Ga's stress field.

\section{Conclusions}
\label{sec:concl}

We presented a new class of stabilized finite elements, based on augmenting the Galerkin formulation with least squares terms, for the MHM method applied to isotropic elasticity.
%\begin{itemize}
%\item
It allows for the use of heterogeneous shear modulus $G \in W^{1,\infty}(\calT_h)$ and Poisson's ratio $\nu \in L^\infty(\OO)$.
%\item
The well definition of the finite elements relies on the choice of positive stabilization parameters $\alpha_K$, for $K \in \calP$, depending only on the polynomial order $k$, $\ess\inf_{\xx\in K}G(\xx)$, and $\|G\|_{W^{1,\infty}(\calT_h)}$.
%\end{itemize}
%
We proved the resulting MHM method is Poisson locking-free.
%Moreover,
%\begin{itemize}
%\item
The displacement and pressure approximations converge at optimal rates, in their natural norms, to the exact solution, and the rates are independent of $\nu$.
Consequently, the stress convergence rate in the $L^2$-norm does not depend on $\nu$ either.
%\end{itemize}
%
We verify numerically the theoretical convergence rates and locking-free property.

The numerical tests showed that the low-order MHM-Ga method (two-level MHM method with Galerkin finite element) loses precision when the Poisson's ratio gets closer to $1/2$.
% In particular, the stress convergence rates in the low-order MHM-Ga method was more affected by the Poisson's locking than the displacement convergence rates.
Besides that, the MHM-Ga solutions were more accurate on our example than the Galerkin solutions for equivalent mesh sizes.
When compared to the MHM-GaLS method (two-level MHM method with stabilized finite element), we observed an accuracy improvement on all measured errors and, in particular, the stress was significantly better approximated.

\appendix

\section{Auxiliary results}
\label{sec:auxiliary}

We start presenting	 two lemmas that are auxiliary in the proof of \cref{sec:wellPosednessTph}.

\begin{lemma}\label{th:auxiliarAlphaLemma}
For every $\uu_h \in \tildeVV_h(K)$,
\begin{align}\label{eq:auxiliarAlphaLemma}
\sum_{\tau\in\calT_h^K} h_\tau^2\|\bdiv\,(2G\,\suh)\|_{0,\tau}^2
\leqslant C_I^{-1}\, \|2G\|_{1,\infty,\calT_h^K}^2 \|\EE(\uu_h)\|_{0,K}^2\,.
\end{align}
\end{lemma}
\begin{proof}
Notice that
\begin{align*}
\|\bdiv\,(2G\,\suh)\|_{0,\tau} &= \| \suh \nabla(2G) + 2G\,\bdiv(\suh)\|_{0,\tau}\\
&\leqslant \| \suh\|_{0,\tau} \|\nabla(2G)\|_{L^\infty(\tau)} + \|2G\|_{L^\infty(\tau)}\|\bdiv(\suh)\|_{0,\tau}\\
&\leqslant \|2G\|_{1,\infty,\calT_h^K} \left(\frac{1}{h_K^2}\|\suh\|_{0,\tau}^2 + \|\bdiv\,\suh\|_{0,\tau}^2\right)^{\frac12}\,,
\end{align*}
where we used \cref{eq:normW1infty}.
Then, we can then use \cref{eq:ineqCI} to obtain
\begin{align*}
\nonumber \sum_{\tau\in\calT_h^K} h_\tau^2 \|\bdiv\,(2G\,\suh)\|_{0,\tau}^2
&\leqslant \|2G\|_{1,\infty,\calT_h^K}^2 \sum_{\tau\in\calT_h^K} \frac{h_\tau^2}{h_K^2}\|\suh\|_{0,\tau}^2 + h_\tau^2\,\|\bdiv\,\suh\|_{0,\tau}^2\\
&\leqslant C_I^{-1}\, \|2G\|_{1,\infty,\calT_h^K}^2 \|\suh\|_{0,K}^2\,.
\end{align*}
\end{proof}

\begin{lemma}\label{th:auxiliarAlphaLemma2}
There exists $C_{0,K} > 0$ depending only on $k$, $d$ and $\|2G\|_{1,\infty,\calT_h^K}$ such that, for every $(\uu_h,p_h) \in \tildeVV_h(K) \times Q_h(K)$,
\begin{align*}
\alpha_K\sum_{\tau\in\calT_h^K} h_\tau^2\|\bdiv\,(2G\,\suh-p_hI)\|_{0,\tau}^2
\leqslant 4 G_{0,K}\,\left(\|\suh\|_{0,K}^2 + C_{0,K}\,\|p_h\|_{0,K}^2\right)
\end{align*}
\end{lemma}
\begin{proof}
Since
\begin{align*}
\|\bdiv\,(2G\,\suh-p_hI)\|_{0,\tau}^2
%&= \|\bdiv\,(2G\,\suh)\|_{0,\tau}^2 + \|\nabla p_h\|_{0,\tau}^2 + 2 \|\bdiv\,(2G\,\suh)\|_{0,\tau}\|\nabla p_h\|_{0,\tau}\\
&\leqslant 2(\|\bdiv\,(2G\,\suh)\|_{0,\tau}^2 + \|\nabla p_h\|_{0,\tau}^2)\,,
\end{align*}
we use \cref{th:auxiliarAlphaLemma}, \cref{eq:alphaHyp}, and a classical inverse inequality
\begin{align*}
C \sum_{\tau\in\calT_h^K} h_\tau^2 \|\nabla p_h\|_{0,\tau}^2
\leqslant \|p_h\|_{0,K}^2\,,
\end{align*}
(see \cite[Lemma 1.138]{ernGuermondFEM})
to obtain
\begin{align*}
&\alpha_K\sum_{\tau\in\calT_h^K} h_\tau^2\|\bdiv\,(2G\,\suh-p_hI)\|_{0,\tau}^2\\
&\qquad\qquad\leqslant 4\, G_{0,K}\,\left(\|\EE(\uu_h)\|_{0,K}^2 + C^{-1}\,\frac{1}{\|2G\|_{1,\infty,\calT_h^K}^2}\, C_I\, \|p_h\|_{0,K}^2\right)
\end{align*}
Therefore, the proof is complete with $C_{0,K} := C_I/\left(C\,\|2G\|_{1,\infty,\calT_h^K}^2\right)$.
%The constant $C_{0,K}$ is the one from the inverse inequality on $Q_h(K)$, i.e.,
%\begin{align}
%C_{0,K} \sum_{\tau\in\calT_h^K} h_\tau^2 \|\nabla p_h\|_{0,\tau}^2
%\leqslant \|p_h\|_{0,K}^2\,,
%\end{align}
%for all $p_h \in Q_h(K)$.
%The constant $C_{0,K}$ depends only the polynomial order $k$ and $d$ (see \cite[Lemma 1.138]{ernGuermondFEM}).
\end{proof}

The following result is auxiliary in the proof of the stability of $B_K$.

\begin{lemma}\label{preStabilityLemma}
There are positive constants $C_1$ and $C_2$ such that, for all $p_h \in Q_h(K)$, there exists $\tilde\vv_h \in\tildeVV_h(K)$ satisfying
\begin{align}\label{eq:preStabilityLemma}
\frac{1}{\|\tilde\vv_h\|_{1,K}}\int_K p_h\,(\Div\,\tilde\vv_h) \dif x \;\geqslant\; C_1 \|p_h\|_{0,K} - C_2 |p_h|_{h,K}\,.
\end{align}
%and a function $\tilde\vv_h \in\tildeVV_h(K)$
%Moreover, there is always a function $\tilde\vv_h \in\tildeVV_h(K)$
\end{lemma}
\begin{proof}
Given $p_h \in Q_h(K)$, the proof of \cite[Lemma 3.3]{stabilizedElasticity} exhibit a function $\vv_h \in \VV_h(K) := \{ \vv_h^{}\in C^0(K): \vv_h^{}\sVert[0]_{\tau}^{}\in\PP_k^{}(\tau)^d\,,\forall\, \tau\in \calT_h^K\}$ satisfying
\begin{align}\label{eq:preStabilityLemmaAux1}
\frac{1}{\|\vv_h\|_{1,K}}\int_K p_h\,(\Div\,\vv_h) \dif x \;\geqslant\; C_1 \|p_h\|_{0,K} - C_2 |p_h|_{h,K}\,,
\end{align}
where the constants $C_1, C_2$ do not depend on $p_h \in Q_h(K)$.
Since every function $\vv_h \in \VV_h(K)$ can be written (uniquely) as $\vv_h = \tilde\vv_h + \vvrm$ where $\tilde\vv_h \in \tildeVV_h(K)$, $\EE(\vvrm) = 0$ and $\int_K \tilde\vv_h\cdot\vvrm \dif x = 0$, it holds
\begin{align*}%\label{preStabilityLemmaAux1}
\frac{1}{\|\vv_h\|_{1,K}}\int_K p_h\,(\Div\,\vv_h) \dif x
&=
\frac{1}{\sqrt{\|\vv_h\|_{0,K}^2 + \|\nabla \vv_h\|_{0,K}^2}}\int_K p_h\,(\Div\,\vv_h) \dif x \nonumber\\
&\leqslant
\frac{1}{\sqrt{\|\tilde\vv_h\|_{0,K}^2 + \|\vvrm\|_{0,K}^2 + \|\EE(\vv_h)\|_{0,K}^2}}\int_K p_h\,(\Div\,\vv_h) \dif x \nonumber\\
&\leqslant
\frac{1}{\sqrt{\|\tilde\vv_h\|_{0,K}^2 + \|\EE(\tilde\vv_h)\|_{0,K}^2}}\int_K p_h\,(\Div\,\tilde\vv_h) \dif x \nonumber\\
&\leqslant
\frac{C_{korn,K}}{\|\tilde\vv_h\|_{1,K}}\int_K p_h\,(\Div\,\tilde\vv_h) \dif x\,,
\end{align*}
for all $\vv_h \in \VV_h(K)$, where we also used Korn's inequality for the space $\tildeVV(K)$ from \cite{10.1093/imanum/drac041}.
Thus, we obtain \cref{eq:preStabilityLemma} by replacing the last expression into \cref{eq:preStabilityLemmaAux1}.
\end{proof}

% \section*{Acknowledgments}
%The authors acknowledge Fabrice Jaillet for providing the L-shaped meshes used in \cref{subsec:analytical_cases}, and the National Laboratory for Scientific Computing (LNCC/MCTI, Brazil) for providing HPC resources of the SDumont supercomputer, which have contributed to the research results reported within this paper. URL: \url{http://sdumont.lncc.br}

\bibliographystyle{siamplain}
\bibliography{references}

\end{document}